\documentclass[11pt]{article}

\usepackage[T1]{fontenc}
\usepackage[utf8]{inputenc}
\usepackage{lmodern,amsmath,amssymb,amsthm,mathtools,booktabs,array}
\usepackage[a4paper,margin=25mm]{geometry}
\usepackage{microtype,needspace}
\usepackage[hidelinks]{hyperref}
\hypersetup{pdftitle={The Conway--Wales lattice and the Rudvalis and Tits groups},pdfauthor={Gerald H\"ohn},pdfsubject={Self-dual even code, integral reflection, derived frame exchange, canonical-frame counting, and simplicity}}
\usepackage{fancyhdr}
\numberwithin{equation}{section}
\newtheorem{theorem}{Theorem}[section]
\newtheorem{proposition}[theorem]{Proposition}
\newtheorem{lemma}[theorem]{Lemma}
\newtheorem{corollary}[theorem]{Corollary}
\theoremstyle{definition}\newtheorem{definition}[theorem]{Definition}
\theoremstyle{remark}\newtheorem{remark}[theorem]{Remark}
\DeclareMathOperator{\Aut}{Aut}
\DeclareMathOperator{\Stab}{Stab}
\DeclareMathOperator{\pr}{pr}
\DeclareMathOperator{\wt}{wt}
\newcommand{\GT}{\sigma}
\newcommand{\id}{1} 
\newcommand{\Z}{\mathbb Z}
\newcommand{\R}{\mathbb R}
\newcommand{\C}{\mathbb C}
\newcommand{\F}{\mathbb F}
\newcommand{\cD}{\mathcal D}
\newcommand{\cC}{\mathcal C}
\newcommand{\cH}{\mathcal H}
\newcommand{\Ru}{\mathrm{Ru}}
\newcommand{\cF}{\mathcal F}
\newcommand{\cX}{\mathcal X}

\newcommand{\DD}{\Delta_1}
\title{The Conway--Wales lattice and the Rudvalis and Tits groups}
\author{Gerald H\"ohn\\[2pt]
  {\small Kansas State University}\\
  {\small\href{mailto:gerald@monstrous-moonshine.de}{\texttt{gerald@monstrous-moonshine.de}}}}
\date{27 September 2026}
\begin{document}
\maketitle
\begin{abstract}
We construct the Conway--Wales lattice from a self-dual even code $C$ over
$\Z_4$ whose binary layers interweave the Hamming codes $\cH_7$ and $\cH_8$.
Explicit signed symmetries lift every automorphism of the binary residue,
and the full signed code group is the coordinate-frame stabilizer
of the Gaussian lattice. An integral reflection in an embedded scaled
self-dual sublattice supplies an additional symmetry, from which the
classical frame exchange is recovered. Counting intrinsically defined
frames determines the full Gaussian isometry group, of order
$583\,704\,576\,000$. Its quotient by the four Gaussian scalars is the
Rudvalis simple group, and the derived subgroup of a cross stabilizer
is the Tits simple group of order $17\,971\,200$. The construction,
automorphism-group determination, and both simplicity proofs are carried
out within the lattice, without assuming the known group orders or using
a recognition or classification theorem.
\end{abstract}

\section{Introduction}
We study the Conway--Wales lattice as a positive-definite Hermitian lattice
of rank $28$ over $\Z[i]$. Its underlying real lattice is unimodular even of
rank $56$. Throughout, a lattice automorphism means a $\Z[i]$-linear isometry
of the specified Hermitian form, so the Gaussian structure is part of the
lattice data.

The point of view is analogous to Conway's treatment of the Leech lattice
\cite{ConwayLeech}: the lattice is the primary object, and its symmetry
group is derived from its geometry. A quaternary lift of two interwoven
Hamming codes gives the lattice and its coordinate-frame stabilizer;
an integral reflection supplies the additional symmetry. An intrinsic
frame count then determines the full Gaussian isometry group and its order.
The projective quotient and the derived subgroup of a cross stabilizer
are proved simple, giving the Rudvalis and Tits groups within the classical
Conway--Wales realization.

Rudvalis proposed the group as a rank-three permutation group and
announced it in 1973 \cite{Rud73}; Conway and Wales gave the first
construction that year \cite{CW}. Conway credits M.~J.~T.~Guy
with finding the coordinate frame that led to the quaternionic
construction \cite[pp.~69--70]{Conway}; see also
\cite[p.~535]{Wilson}. Our explicit construction uses the two seeds reproduced by Duncan
\cite[\S5.1, Tables~1--2]{DuncanII}, expressed below by quadratic-residue
formulas; the original arrays and the additional transformation are in
\cite[Tables~1--3, \S5]{Conway}.
Wilson's account \cite[\S1]{Wilson} describes the associated
Rudvalis--Tits geometry. We reconstruct from the code the classical
quartet--dozen--base incidence described by Conway \cite[pp.~76--79]{Conway}.

\paragraph{Notation.}
The Hamming codes are $\cH_7$ and $\cH_8$. The quaternary code is $C$;
its binary residue and torsion codes are
$\cD=C\bmod2$ and $\cC=\{a\in\F_2^{56}:2a\in C\}$, respectively.
The constructed lattice group is $G$, and the coordinate-frame stabilizer
is $H$. The coordinate frame is $F$ and the second frame is $F'$.
Affine coordinates are $(p,\,e)$, four-block labels are functionals
$\ell\in E^*$, and block translations are indexed by
$\tau\in\operatorname{End}(E)$. A couple is denoted by $\kappa$;
its dozen is $\Delta(\kappa)$, and the canonical frame of a quartet $Q$
is $F(Q)$.
The row kernel of $H$ is $P_0$, the Frobenius operator is $\GT$, and
$R_0$ is the specified monomial row involution. Superscript $\top$
denotes transpose. Supports are denoted by $\Sigma$; $\Omega$ is reserved
for the five quartets through the reference couple. Single indices
$X_1$, $\ldots$, $X_{28}$ number the coordinate crosses, with four consecutive
indices per row.
We write $I$ for the prescribed orthogonal complex structure on
$\R^{56}$, with $I^2=-\id$, and $\id$ for the identity operator on the
space in question. The scalar $i\in\C$ acts as $I$ when real coordinates
are used; bold $\mathbf1$ is reserved for all-one words and indicator
functions, not identity maps.
In particular, the Gaussian isometry group is the centralizer of $I$
in the real lattice isometry group. The group $G$ below is this
Gaussian group.

The seven blocks form a Fano plane and each carries an affine $\cH_8$;
their incidence determines the binary residue and torsion layers.
Two integer seeds specify the quaternary lift, whose self-duality follows
from their correlations and binary span. The signed symmetries lift the
binary group: conjugates of one rank-one translation, together with the
complex structure, generate the block kernel, and the sign kernel then
determines the signed group and its order.
\begin{theorem}\label{thm:main}
The code $C$ defined in \eqref{eq:seedcode} is self-dual and even, with
$C\cong\Z_4^{19}\oplus\Z_2^{18}$ and minimum Euclidean weight $16$.
Its full signed automorphism group $H$ has order $2^{13}\cdot168$,
with the concrete generators of Proposition~\ref{prop:H0}.
Every element of $H$ commutes with the signed permutation $I$ satisfying
$I^2=-\id$.

The coordinate lift in \eqref{eq:lattice}, with complex multiplication
given by $I$, defines the Gaussian lattice $L$. Its underlying real lattice is
unimodular even of rank $56$ and minimum $4$, and $H$ is its full
coordinate-frame stabilizer. The heavy seed gives the second frame
$F'$ in \eqref{eq:secondframe}. The three blocks of its $\GT$-stable
Fano line $\lambda$ contain a sublattice $\Lambda_\lambda$ with
$\Lambda_\lambda^*=\tfrac12\Lambda_\lambda$, giving the integral reflection
$r_*=2\pr_{V'_\lambda}-\id$.
The comparison map $S$ in \eqref{eq:S} satisfies
$S=-E_2(E_2r_*)^3$ and is therefore a complex-linear lattice isometry, with
\[
 S^2=\id,\qquad SP=P^{-1}S,\qquad S(F)\cap F=\varnothing.
\]
\end{theorem}

A \emph{cross} is the set of four Gaussian-unit multiples of a minimal
vector. A \emph{frame} consists of $28$ mutually Hermitian-orthogonal crosses.
Let $F$ be the coordinate frame. The code calculation proves that $H$ is its
full stabilizer. Put $G=\langle H,\,r_*\rangle=\langle H,\,S\rangle$.
\begin{theorem}\label{thm:groupintro}
The lattice has $4060$ minimal crosses. The group $G$ is transitive on
crosses, ordered orthogonal pairs, and intrinsic quartets. There are
$424125$ canonical frames, forming one $G$-orbit; each orthogonal couple lies in $45$ of
them and each cross in $2925$. Every couple defines an intrinsic
reflection in $G$. The group $G$ is generated by $H$ and the reflection
in the seed couple $\{X(B/2),\,X(J_1 B/2)\}$.
The operator $-S$ is also the reflection in an intrinsic dozen. Moreover
\[
 G=\Aut_{\Z[i]}(L,\,h),\qquad
 |G|=424125\,(2^{13}\cdot168)=583\,704\,576\,000.
\]
Its centre is $\mu_4$, and the quotient $R=G/\mu_4$ is nonabelian simple of
order $145\,926\,144\,000$. For a cross $x$, the stabilizer $R_x$ has
order $35\,942\,400$ and a nonabelian simple derived subgroup of index two
and order $17\,971\,200$. If $f\in x$, then $G_x=\mu_4\times G_f$
and $G_f\cong R_x$, where $G_f$ is the absolute vector stabilizer.
\end{theorem}

\paragraph{The reflection construction.}
The signed symmetries give a second orthogonal frame from the heavy seed.
A six-vector relation embeds the glue generator of a binary self-dual code
of length $24$; its orbit then gives the three-block sublattice $\Lambda$.
The inclusion $2\Lambda^*\subseteq\Lambda$ makes the associated reflection
integral. Self-adjointness, the trace, and a three-block overlap identity
recover the classical frame exchange $S$ from two reflections.
Both $r_*$ and $-S$ are subsequently identified as reflections in dozens.

\paragraph{The counting strategy.}
The code determines the minimal shell and the coordinate-frame stabilizer
$H$. At a coordinate cross $x$, the orthogonal crosses form four
$H_x$-orbits of sizes $3$, $24$, $192$, $1536$. A move joining the two coordinate
orbits gives possible enlarged orbit sizes $27$, $219$, $1563$, $1755$.
The frame-stabilizer calculation excludes $27$, while the primes $73$
and $521$ in the next two sizes are incompatible with the rank-$56$
integral representation. This proves transitivity on the orthogonal
shell, and hence on crosses and couples.

The remaining input is local geometry: five quartets through a couple,
a canonical frame determined by each quartet, and seven quartets determining
each frame. Counting these incidences gives
\[
 \frac{4060\cdot1755}{2}
 \;\frac{5}{6}\;\frac{1}{7}=424125.
\]
Here $2$ forgets the ordering of a couple, $6$ counts couples in a quartet,
and $7$ is the fibre size of the canonical-frame construction, not the
number of all quartets contained in a frame. Transitivity on this intrinsic
family and its stabilizer $H$ give $G=\Aut_{\Z[i]}(L,\,h)$ and
$|G|=424125\,|H|$, completing the analogy with the Leech-lattice order argument.

Section~\ref{sec:framecode} constructs the binary geometry, determines the
signed lift and group, and counts its least-weight words in coordinates
suited to the lattice. Section~\ref{sec:construction} builds the second
frame and integral reflection before recovering $S$.
Section~\ref{sec:geometryorder} proves orbit fusion, local quartet geometry,
and the canonical-frame count. Sections~\ref{sec:simplicity}
and~\ref{sec:titsmain} give the two simplicity proofs.

\section{The quaternary code and its signed automorphisms}\label{sec:framecode}
\subsection{The Fano plane and its eight-point blocks}\label{sec:hammingweave}
Let $E$ be a three-dimensional vector space over $\F_2$. Its projective
plane is canonically $\Pi=E\setminus\{0\}$: the Fano lines are the triples
$\{p,\,q,\,p+q\}$ for distinct nonzero $p$, $q$, so incidence is determined by
addition in $E$. Conversely, adjoining $0$ to an abstract Fano
plane recovers $E$, with $p+q$ the third point on the line through distinct
$p$, $q$. The coordinates $(p,\,e)\in\Pi\times E$ may be viewed as the affine
line $\{e,\,e+p\}$ with endpoint $e$ marked and direction $p$.
There are four affine lines in each of the seven directions, hence $28$
unmarked lines and $56$ marked lines. The endpoint switch
$\iota(p,\,e)=(p,\,e+p)$ is a fixed-point-free involution whose $28$ pairs
will be the Gaussian coordinate directions. These two-point affine lines
are distinct from the three-point Fano lines. The complex structure will
be a signed lift of $\iota$ with square $-\id$.

The two Hamming codes are
\[
 \cH_7=\{s\in\F_2^\Pi:\textstyle\sum_p s(p)p=0\},\qquad
 \cH_8=\{e\mapsto a+\ell(e):a\in\F_2,\ \ell\in E^*\}.
\]
They have parameters $[7,\,4,\,3]$ and $[8,\,4,\,4]$. Their geometric generators
are the Fano lines and the affine planes, respectively; $\cH_7^\perp$
consists of the linear functions on $\Pi$. Set
\begin{equation}\label{eq:hammingweave}
 \cD_0=(\F_2^\Pi\otimes\langle\mathbf1_8\rangle)
             +(\cH_7^\perp\otimes\cH_8).
\end{equation}
Equivalently, it is generated by full blocks and the rectangles
$(\Pi\setminus\lambda)\times A$, with $\lambda$ a Fano line and $A$ an
affine plane. Its words are uniquely $\epsilon_p+(Mp)\cdot e$, where
$\epsilon\in\F_2^\Pi$ and $M\in\operatorname{Hom}(E,\,E^*)$; hence
$\dim\cD_0=7+9=16$.

The remaining gluing also comes from incidence. For each $p$, choose a
Fano line $\lambda_p$ not through $p$ and put $T_p=\{p\}\cup\lambda_p$.
The point-sum map from even subsets of $T_p$ to $E$ is an isomorphism:
no pair has sum zero, and $\sum_{e\in T_p}e=p\ne0$.
Let $S_p(c)$ be the unique even subset with sum $c$, and let
$\theta_p(c)$ be its indicator. Thus $S_p(0)=\varnothing$, $S_p(p)=T_p$,
and otherwise $S_p(c)$ is the unique pair in $T_p$ with sum $c$.
The map $c\mapsto\theta_p(c)$ is linear.

An even subset of $E$ has indicator of Boolean degree at most two.
Its point sum is its quadratic class: modulo $\cH_8$, the class $c$
has polar form $\det(c,\,e,\,e')$. In particular our incidence sections satisfy
\begin{equation}\label{eq:hammingsections}
\begin{aligned}
 \theta_p(c)(0)&=0,\\
 \theta_p(c)(e+e')+\theta_p(c)(e)+\theta_p(c)(e')&=\det(c,\,e,\,e'),\\
 \theta_p(c)(c)&=\mathbf1_{\{p=c\}}.
\end{aligned}
\end{equation}
We may use any linear sections with these properties. For later
calculations, in a basis of $E$ write
\begin{equation}\label{eq:eta}
 q_c(e)=c_1e_2e_3+c_2e_1e_3+c_3e_1e_2,\qquad
 \theta_p(c)(e)=q_c(e)+\eta_c(p)\cdot e.
\end{equation}
Every invertible matrix over $\F_2$ has determinant one, so the polar
form is independent of the chosen basis.

\begin{lemma}\label{lem:hammingsections}
Any two normalized choices give equivalent binary codes
\begin{equation}\label{eq:hammingglue}
 \cD=\cD_0+\{(\theta_p(c)(e))_{p,\,e}:c\in E\},\qquad
 \cC=\cD^\perp,
\end{equation}
by changing the affine origin in each block. Every Fano collineation
therefore lifts to an affine coordinate permutation of $\cD$.
\end{lemma}
\begin{proof}
Two sections differ by a bilinear form in $c$, $e$ with zero diagonal,
hence by $\det(c,\,t_p,\,e)$ for a unique $t_p\in E$. Translation by $t_p$
adds this term and the block constant $\theta_p(c)(t_p)$, which is
absorbed in $\cD_0$. Simultaneous application of $a\in\operatorname{GL}(E)$
to $p$, $c$, $e$ preserves \eqref{eq:hammingsections}; a change of origins
therefore identifies the resulting sections with the original ones.
\end{proof}

\begin{proposition}\label{thm:binarynormal}
The codes $\cD$, $\cC$ have dimensions $19$, $37$, with $\cD\subseteq\cC$.
The code $\cD$ is doubly even, contains $\mathbf1_{56}$, and its words are
uniquely
\begin{equation}\label{eq:binarynormal}
 d(p,\,e)=\epsilon_p+(Mp+\eta_c(p))\cdot e+q_c(e),\qquad
 \epsilon\in\F_2^\Pi,\quad M\in\operatorname{Hom}(E,\,E^*),\quad c\in E.
\end{equation}
The support in every block has point sum $c$, and
$\cD/\cD_0\cong\mathcal E_8/\cH_8\cong E$, where $\mathcal E_8$ is
the even-weight code on $E$.
\end{proposition}
\begin{proof}
The common point sum recovers $c$, then the linear terms recover $M$
and the constants recover $\epsilon$. Two identities will be used throughout:
\begin{align}
 q_c(c)+c\cdot\eta_c(p)&=\mathbf1_{\{p=c\}},\label{eq:balancedidentity}\\
 \sum_e\theta_p(c)(e)\theta_p(d)(e)&=\det(c,\,d,\,p).\label{eq:productidentity}
\end{align}
The first is the normalization. For the second use the incidence sections:
the six pair sums in $T_p$ are $E\setminus\{0,\,p\}$, and two distinct
pairs meet in one point exactly when their sums together with $p$ are
independent. The cases involving $0$, $p$, or equal sums give zero on both
sides. Translating sections and adding constants preserves these pairings,
since all sections have even weight.

Two affine functions have product summing to zero on a block. The pairing
of $\theta_p(c)$ with $(Mp)\cdot e$ is $c\cdot Mp$; two quadratic terms
are handled by \eqref{eq:productidentity}. Summing over $p$ gives zero,
since $\sum_p p=0$. Thus $\cD$ is self-orthogonal and $\dim\cC=37$.
Since $\mathbf1_{56}\in\cD$, the code $\cC$ is even. Block indicators and
rank-one bilinear words have weights eight and sixteen. The incidence
word $(\theta_p(c))_p$, for $c\ne0$, has pattern $(4,\,2^6)$ and weight
sixteen. These words span $\cD$; self-orthogonality makes their span
doubly even. Equivalence of normalized sections proves the general case.
\end{proof}
Every linear block translation preserves $\cD$: for $\tau\in\operatorname{End}(E)$,
\begin{equation}\label{eq:tracetranslation}
 \theta_p(c)(e+\tau p)+\theta_p(c)(e)
     =\det(c,\,\tau p,\,e)+\theta_p(c)(\tau p).
\end{equation}
The first term is bilinear in $p$, $e$, and the second is a block constant.

\subsection{Binary weights and automorphisms}\label{sec:binaryaut}

\begin{theorem}[The least residue weights]\label{thm:WD}
The only nonzero words of $\cD$ of weight less than $16$ are the seven
block indicators, of weight eight. The weight-$16$ words have block
patterns and multiplicities
\[
 (8,\,8,\,0^5):21,\qquad (4^4,\,0^3):784,\qquad(4,\,2^6):896.
\]
\end{theorem}
\begin{proof}
For $c=0$ in \eqref{eq:binarynormal}, a matrix $M$ of rank $r$ has
$2^{3-r}-1$ zero blocks and $8-2^{3-r}$ balanced blocks.
The former have weights zero or eight, and the latter weight four.
Thus a word below weight sixteen is a block indicator. At weight
sixteen, $M=0$ gives two blocks, in $\binom72=21$ ways;
otherwise $M=\ell'\ell^{\top}$ has rank one, giving
$7\cdot7\cdot2^4=784$ words.

For $c\ne0$, a block is balanced exactly when the quadratic function
has difference one in its radical direction $c$. By
\eqref{eq:balancedidentity}, the balanced-block indicator is
\[
 p\longmapsto\mathbf1_{\{p=c\}}+(M^{\top}c)\cdot p.
\]
Its possible weights are $1$, $3$, $5$. Every unbalanced block has weight
at least two, so the total weight is at least sixteen, with equality
only for one balanced block and six blocks of weight two.
This requires $M^{\top}c=0$, giving $2^6$ matrices, and there are two
choices of the heavy-block constant. Thus there are $128$ words for
each of the seven nonzero classes $c$.
\end{proof}
The supports of pattern $(4,\,2^6)$ are called \emph{heavy}; their unique
weight-four block is the \emph{heavy block}. A signed word on a support
has entries $\pm1$ there and zero elsewhere.

\begin{proposition}[The torsion code]\label{prop:torsionmin}
The puncturing of $\cD$ onto any block is the full even-weight
binary code of length eight. The code $\cC=\cD^\perp$ has no nonzero
word of weight at most $4$.
\end{proposition}
\begin{proof}
The constant, linear coefficients, and quadratic class in a block
can be prescribed arbitrarily, proving the puncturing assertion.
By duality, the shortening of $\cC$ to a block is the repetition code.
Orthogonality to the block indicators forces even weight in every block,
while a nonzero word supported in one block has weight eight.
Thus a word of positive weight at most four would select two distinct
positions in each of two blocks $p\ne q$; call those in block $p$ $e$, $e'$.
Choose $\ell\in E^*$ with $\ell(p)=1$ and $\ell(q)=0$.
Orthogonality to every word $\ell(p')a(e'')\in\cD_0$, $a\in E^*$,
then gives $a(e+e')=0$ for all $a$, a contradiction.
\end{proof}

\paragraph{The full binary automorphism group.}
The Hamming gluing determines the full permutation group directly. Automorphisms
of $\cD$ in this subsection are coordinate permutations.
\begin{proposition}\label{prop:binaryaut}
Every automorphism of $\cD$ permutes the seven blocks as an element of
$\operatorname{GL}_3(2)$. The subgroup fixing all seven blocks consists exactly
of the translations
\[
 (p,\,e)\longmapsto(p,\,e+\tau p),\qquad \tau\in\operatorname{End}(E).
\]
Writing $\Gamma=\Aut(\cD)$ and $N$ for this translation subgroup, one has
\[
 1\longrightarrow N\longrightarrow\Gamma
 \longrightarrow\operatorname{GL}(E)\longrightarrow1,
 \qquad |\Gamma|=2^9\cdot168.
\]
In particular, $\Gamma$ is doubly transitive on the seven blocks.
\end{proposition}
\begin{proof}
The seven weight-eight words recover the blocks. The active block sets of
weight-sixteen words having four nonzero blocks are precisely the complements
of the seven Fano lines. Hence block permutations preserve the Fano plane,
whose permutation group is $\operatorname{GL}_3(2)$: a line is exactly a triple
$\{p,\,q,\,p+q\}$, so its permutations extend additively after fixing zero.

Let an automorphism fix each block. The subcode $\cD_0$ is intrinsic: it is
spanned by block indicators and the weight-sixteen words with four active
blocks. The latter give every rank-one bilinear term, and these span all
bilinear terms. The restriction of $\cD_0$ to a block is the affine-function
code $\cH_8$. A permutation preserving this code is affine,
since its three coordinate functions must be affine and independent.
Write its action on block $p$ as $e\mapsto A_pe+t_p$.

An even subset with point sum $c$ has point sum $A_pc$ after this
map; the translation contributes zero. Every word of $\cD$ has the same
point sum in all blocks, and every $c\in E$ occurs. Hence $A_pc$ is
independent of $p$ for every $c$, so $A_p=A$ for all $p$.
For $c\ne0$, a heavy word of class $c$ has its heavy block at $c$,
by Theorem~\ref{thm:WD} and \eqref{eq:balancedidentity}. Its image has
class $Ac$ and the same heavy block. Thus $Ac=c$ for all $c$, giving
$A=\id$.

Finally, translation by $(t_p)$ changes the linear part of
$\theta_p(c)$ by $e\mapsto\det(c,\,t_p,\,e)$. Preservation of $\cD$ requires
this linear coefficient to depend linearly on $p$ for every $c$.
For distinct nonzero $p$, $q$, put $\delta=t_p+t_q+t_{p+q}$.
Then $\det(c,\,\delta,\,e)=0$ for all $c$, $e$, which forces $\delta=0$.
Thus $p\mapsto t_p$ extends to a linear map on $E$. Conversely every such
translation preserves $\cD$ by \eqref{eq:tracetranslation}. There are
$2^9$ of them. Every Fano collineation lifts by
Lemma~\ref{lem:hammingsections}, so the image on the blocks is all of
$\operatorname{GL}(E)$. This proves the exact sequence and order.
Any two distinct nonzero vectors of $E$ are independent, giving double
transitivity.
\end{proof}

\subsection{The quaternary lift and its sign kernel}\label{sec:binary}
Let
$\pi:\Z_4^{\Pi\times E}\to\F_2^{\Pi\times E}$ be reduction modulo two.
For a quaternary word $u$ define its \emph{Euclidean weight} by
\[
 \operatorname{wt}_E(u)=n_1(u)+4n_2(u)+n_3(u),
\]
where $n_a$ counts entries equal to $a\in\Z_4$.
A quaternary code is \emph{even} if all these weights are divisible by
eight. Equivalently, every integer lift $\widetilde u$ has normalized
squared norm $\frac14\sum_\omega\widetilde u_\omega^2\in2\Z$;
this condition is unchanged by adding multiples of four to its entries.
Self-duality uses the standard dot product modulo four.

The binary pair specifies the odd positions and the ambiguity between
quaternary lifts of the same residue.
\begin{lemma}\label{lem:binaryduality}
For any self-dual quaternary code with residue $\cD$, its binary
torsion code is $\cC=\cD^\perp$:
\begin{equation}\label{eq:binarylayers}
 \cD=\pi(C),\qquad
 \cC=\{a\in\F_2^{\Pi\times E}:2a\in C\}.
\end{equation}
Consequently there is an exact sequence of additive groups
\begin{equation}\label{eq:binaryexact}
 0\longrightarrow\cC\xrightarrow{\ a\mapsto2a\ }C
   \xrightarrow{\ \pi\ }\cD\longrightarrow0.
\end{equation}
\end{lemma}
\begin{proof}
A binary word $a$ is orthogonal to $\cD$ exactly when $2a\cdot u=0$
modulo four for every $u\in C$. By self-duality this is equivalent to
$2a\in C$, giving both assertions.
\end{proof}
Two lifts of a residue differ by $2\cC$; $2C=2\cD$. We retain
$\cC$, $\cD$ as in the framed-VOA notation of \cite{DGH}.
For a support $\Sigma$, put $\cC_\Sigma=\cC\cap\F_2^\Sigma$.
Whenever signed lifts on $\Sigma$ exist, they form an affine space over
$\cC_\Sigma$. This is the sign freedom used in the subsequent vector count.

\paragraph{The sign kernel.}
For quaternary codes, $\Aut(C)$ means the full group of signed
coordinate permutations preserving $C$.
\begin{theorem}\label{thm:signkernel}
For every self-dual lift of $\cD$, the full sign kernel consists of
the block-constant patterns indexed by $\cH_7$ and has order $16$.
\end{theorem}
\begin{proof}
For binary words, $*$ denotes coordinatewise multiplication.
Signs $s$ send $u$ to $u-2(s*\pi(u))$. By
Lemma~\ref{lem:binaryduality}, they preserve the code exactly when
$s*d\in\cC$ for all $d\in\cD$, or equivalently when
\[
 s\in\langle d*v:d,\,v\in\cD\rangle^\perp.
\]
Multiplication by a block indicator shows that $s$ is orthogonal on
that block to the puncturing of $\cD$, which is the full even-weight
binary code. Hence $s$ is constant on each block, with values $s_p$.
Products of two affine functions sum to zero on a block. The block
sum of $(Mp)\cdot e$ times $\theta_p(c)(e)$ is $c\cdot Mp$.
Since $c$, $M$ are arbitrary, the necessary condition is
$\sum_p s_pp=0$. This condition is sufficient as well:
\eqref{eq:productidentity} handles the products of two quadratic terms,
and the preceding calculations handle the remaining products.
The kernel is therefore exactly $\cH_7$.
\end{proof}
Denote the sign kernel by $E_C$. For $\ell\in E^*$ write
\[
 (E_\ell v)(p,\,e)=(-1)^{\ell(p)}v(p,\,e).
\]
Thus $E_C=\langle-\id,\,E_1,\,E_2,\,E_4\rangle$; a numerical subscript
uses the binary encoding of the functional $\ell$. These are precisely
the block signs from Theorem~\ref{thm:signkernel}.

\subsection{The signed seeds and their coordinates}\label{sec:seeds}
For the seed construction and the monomial matrices, choose the sections
below and the coordinates of Table~\ref{tab:hamminglabels}.
We encode $E=\F_2^3$ by $0$, $\ldots$, $7$, with the least significant bit
first. The parameter $c\mapsto\eta_c(p)$ is linear, so its values on $1$, $2$, $4$
determine it; all entries are three-bit vectors.
\begin{table}[htbp]\centering\small
\begin{tabular}{c@{\qquad}rrr}\toprule
$p$&$\eta_1(p)$&$\eta_2(p)$&$\eta_4(p)$\\\midrule
1&3&4&3\\
2&6&2&3\\
3&4&5&3\\
4&0&0&7\\
5&6&5&2\\
6&2&5&0\\
7&6&1&1\\\bottomrule
\end{tabular}
\caption{One choice of affine origins and quadratic sections.}\label{tab:eta}
\end{table}
With the convention \eqref{eq:eta}, the table gives
\[
 q_c(c)+c\cdot\eta_c(p)=\mathbf1_{\{p=c\}},\qquad
 \theta_{T_Pp}(T_Pc)(T_Pe)=\theta_p(c)(e),
 \quad T_P=(5,\,3,\,2).
\]
The polar form of $q_c$ is $\det(c,\,e,\,e')$, so these sections satisfy
\eqref{eq:hammingsections}. The two seed residues in
\eqref{eq:traceseedresidues} follow by expressing the seeds in
\eqref{eq:seedarrays} in the signed positions of
Table~\ref{tab:hamminglabels}. These tables specify a coordinate
representative of the abstract Hamming gluing.

\paragraph{The seeds and signed coordinates.}
Choose signed coordinates for the two seeds and the operators below,
using the sections in Table~\ref{tab:eta}. Order the seven blocks as
\[
 (p_0,\,\ldots,\,p_6)=(6,\,1,\,5,\,7,\,4,\,2,\,3).
\]
In each block let $(\mathbf e_0,\,\ldots,\,\mathbf e_7)$ be an ordered coordinate
basis. Table~\ref{tab:hamminglabels} identifies the geometric coordinates
with the signed basis
$(\mathbf e_1,\,-\mathbf e_0,\,\mathbf e_2,\,\mathbf e_3,\,
\mathbf e_4,\,\mathbf e_5,\,\mathbf e_6,\,\mathbf e_7)$.
All products of operators are composed from right to left.

\begin{table}[htbp]\centering\small
\begin{tabular}{c r rrrrrrrr}\toprule
$r$&$p$&\multicolumn{8}{c}{$e$ in the above eight real positions}\\\midrule
0&6&0&6&2&4&7&1&3&5\\
1&1&0&1&3&2&4&5&6&7\\
2&5&0&5&6&3&2&7&1&4\\
3&7&0&7&1&6&3&4&5&2\\
4&4&0&4&5&1&6&2&7&3\\
5&2&0&2&7&5&1&3&4&6\\
6&3&0&3&4&7&5&6&2&1\\\bottomrule
\end{tabular}
\caption{Hamming coordinates. The two positions paired by $I$ are
$e$ and $e+p$.}\label{tab:hamminglabels}
\end{table}
The row shift $(Pv)_r=v_{r-1}$ is used to specify a short list of lifts;
row subscripts in this coordinate realization are taken modulo seven.
It sends $(p,\,e)$ to $(T_Pp,\,T_Pe)$, where the columns of $T_P$ are $(5,\,3,\,2)$.
The $T_P$-invariance verified above shows that $P$ preserves the binary pair. The following signed permutations act on
each eight-tuple:
\begin{equation}\label{eq:signedgenerators}
\begin{aligned}
  I(v)&=(-v_1,\,v_0,\,-v_3,\,v_2,\,-v_5,\,v_4,\,-v_7,\,v_6),\\
 J_1(v)&=(-v_2,\,-v_3,\,v_0,\,v_1,\,v_6,\,v_7,\,-v_4,\,-v_5),\\
 J_2(v)&=(-v_4,\,-v_5,\,-v_6,\,-v_7,\,v_0,\,v_1,\,v_2,\,v_3).
\end{aligned}
\end{equation}
They satisfy $I^2=J_1^2=J_2^2=-\id$,
$IJ_1=J_1 I$, $IJ_2=J_2 I$, and
$J_1J_2=-J_2J_1$. Write
\[
 J_0=\id,\qquad J_3=J_2J_1.
\]
The group $\mathcal J=\langle I,\,J_1,\,J_2\rangle$ consists of the
sixteen operators $\pm I^\epsilon J_s$.
Their coordinate permutations are linear block translations
$e\mapsto e+\tau p$, as read from Table~\ref{tab:hamminglabels}, and therefore
preserve $\cD$, $\cC$ by \eqref{eq:tracetranslation}.
In particular, $I$ is a signed lift of the endpoint switch $\iota$:
on each paired real coordinate plane it acts as $(a,\,b)\mapsto(-b,\,a)$.

Pair positions by $I$ and write $iv=Iv$. Put
$\mathcal Q=\{1,\,2,\,4\}\subset\Z/7\Z$. The complex coordinate basis is
\begin{equation}\label{eq:complexbasis}
 (w_0,\,w_2,\,w_4,\,w_1)=(\mathbf e_1,\,\mathbf e_2,\,\mathbf e_4,\,\mathbf e_6),
\end{equation}
so the eight-tuple $(v_0,\,\ldots,\,v_7)$ has coordinates
$(v_1-iv_0,\,v_2+iv_3,\,v_4+iv_5,\,v_6+iv_7)$.
For real vectors let $(v,\,w)$ be the Euclidean dot product. On a complex
block, and by summing over blocks on the whole space, use the Hermitian
form linear in the second argument
\begin{equation}\label{eq:hermitianform}
 h(v,\,w)=(v,\,w)+i(Iv,\,w)=\overline v^{\top}w.
\end{equation}
The two integer seeds are defined, for $r\in\mathcal Q$, by
\begin{equation}\label{eq:seedarrays}
\begin{aligned}
 B_0&=w_0+w_1+w_2+w_4,&
 A_0&=-iw_0+w_1+w_2+w_4,\\
 B_r&=w_0-w_r,& A_r&=0,\\
 B_{-r}&=w_{2r}-w_r,&
 A_{-r}&=w_0+iw_r-w_{2r}+w_{4r}.
\end{aligned}
\end{equation}
These formulas specify the original eight real entries in each row;
the chosen phase of $w_0$ makes all four complex coordinates of every $B_r$ real.
Their residues display the two parts of the Hamming gluing. With
$\ell=(0,\,1,\,0)\in E^*$, they are
\begin{equation}\label{eq:traceseedresidues}
\begin{aligned}
 A(p,\,e)\bmod2&=\mathbf1_{\{p=6\}}+\ell(p)+\ell(p)\ell(e),\\
 B(p,\,e)\bmod2&=\theta_p(6)(e)+\mathbf1_{\{p\in\{1,\,4,\,5,\,6\}\}}.
\end{aligned}
\end{equation}
The first residue belongs to $\cD_0$, and the second has a nonzero
common quadratic class.

Define the quaternary code by
\begin{equation}\label{eq:seedcode}
 C=2\cC+\big\langle aA,\,aB:a\in\langle P,\,\mathcal J\rangle\big\rangle_{\Z_4}
       \subseteq\Z_4^{56}.
\end{equation}
The real integer coordinates of the seeds are reduced modulo four
in this formula. The summand $2\cC$
specifies the all-even layer, while the displayed lifts specify the signs
above the residue code. The operators $P$ and $\mathcal J$ preserve $C$:
they preserve $2\cC$ and permute the signed seed translates.
All residues lie in $\cD$. We prove self-duality and evenness first,
then determine the signed symmetries needed in the lattice geometry.

Up to sign, the two orbits have the $112$ representatives
$P^rI^\epsilon J_sA$, $P^rI^\epsilon J_sB$ with $0\le r\le6$,
$0\le s\le3$ and $\epsilon=0$ or $1$. Each has squared Euclidean norm $16$.
Their residues generate $\cD$, as proved next.

\subsection{Self-duality of the lift}\label{sec:unimod}
\begin{lemma}[The order-three seed symmetry]\label{lem:frobenius}
Set
\[
 \psi(v_0,\,\ldots,\,v_7)=(v_0,\,v_1,\,v_6,\,v_7,\,v_2,\,v_3,\,v_4,\,v_5),
\]
and define
\begin{equation}\label{eq:frobenius}
 (\GT v)_{2r}=\psi(v_r)\qquad(r\in\Z/7\Z).
\end{equation}
Then $\GT$ is a coordinate permutation of order three fixing $A$, $B$.
It centralizes $I$, satisfies $\GT P\GT^{-1}=P^2$, and cyclically
conjugates $J_1$, $J_2$, $J_3$.
\end{lemma}
\begin{proof}
The map $\psi$ fixes $w_0$ and sends $w_r$ to $w_{2r}$ for
$r\in\mathcal Q$. Thus \eqref{eq:seedarrays} gives
$A_{2r}=\psi(A_r)$ and $B_{2r}=\psi(B_r)$.
The conjugations of $I$, $J_1$, $J_2$, $J_3$ follow from
\eqref{eq:signedgenerators}; the row formula gives the order and
conjugation of $P$. In particular $\GT$ permutes the signed seed translates.
\end{proof}

For a complex block vector $v$, let $\mathsf Q(v)$ be the unique linear
map commuting with $J_1$, $J_2$ and sending $w_0$ to $v$. Its columns
in the basis \eqref{eq:complexbasis} are
$v$, $-iJ_1 v$, $-iJ_2 v$, $-iJ_3v$: the quaternion relations show that this
matrix commutes with both operators, and commutation determines every
column. In particular $\mathsf Q(w_0)=\id$.
Put $K_a=\mathsf Q(w_a)$ for $a\in\mathcal Q$.
The defining conditions make $\mathsf Q$ linear, so
\[
 \mathsf Q\!\left(v_0w_0+\sum_{a\in\mathcal Q}v_aw_a\right)
       =v_0\id+\sum_{a\in\mathcal Q}v_aK_a.
\]
The $K_a$ are traceless Hermitian involutions, and the signed formulas give
\begin{equation}\label{eq:Qblock}
 K_a^2=\id,\qquad K_aK_{2a}=-iK_{4a}\quad(a\in\mathcal Q).
\end{equation}
Writing $\mathsf Q_k=\mathsf Q(B_k)$, the seed formulas become
\begin{equation}\label{eq:Pauliblocks}
 \mathsf Q_0=\id+K_1+K_2+K_4,\qquad
 \mathsf Q_a=\id-K_a,\qquad
 \mathsf Q_{-a}=K_{2a}-K_a\quad(a\in\mathcal Q).
\end{equation}

\begin{lemma}[The difference-set correlation]\label{lem:differencecorrelation}
For every $m\in\Z/7\Z$ one has
\begin{equation}\label{eq:blockcorrelation}
 \sum_{k=0}^6\mathsf Q_k\mathsf Q_{k+m}=16\delta_{m,\,0}\id.
\end{equation}
\end{lemma}
\begin{proof}
Work with a central variable $t$ satisfying $t^7=1$, and set
\[
 d(t)=1+t+t^2+t^4,\quad n(t)=1+t+\cdots+t^6,\quad
 u(t)=\sum_{a\in\mathcal Q}(t^{3a}-t^{-3a})K_a.
\]
The set $\{0,\,1,\,2,\,4\}$ is a $(7,\,4,\,2)$ difference set: each nonzero
residue occurs twice among its ordered differences. Thus
$d(t)d(t^{-1})=2+2n(t)$ and $n(t)^2=7n(t)$.
The coefficient of $K_a$ in $\sum_k\mathsf Q_kt^k$ is
\[
 1-t^a-t^{-a}+t^{3a}=d(t)(t^{3a}-t^{-3a}),
\]
so $\widehat{\mathsf Q}(t):=\sum_k\mathsf Q_kt^k=d(t)(\id+u(t))$.
The $K_a$ square to $\id$ and anticommute in pairs. Since the six
residues $\pm6a$, $a\in\mathcal Q$, exhaust the nonzero residues,
\[
 u(t^{-1})=-u(t),\qquad
 u(t)^2=\sum_{a\in\mathcal Q}(t^{6a}+t^{-6a}-2)\id=(n(t)-7)\id.
\]
Consequently
\[
 \widehat{\mathsf Q}(t)\widehat{\mathsf Q}(t^{-1})
   =(2+2n(t))(8-n(t))\id=16\cdot\id.
\]
Comparing the coefficient of $t^{-m}$ proves the assertion.
\end{proof}

\begin{lemma}\label{lem:codeorthogonal}
The code $C$ is self-orthogonal and even.
\end{lemma}
\begin{proof}
For seeds $v$, $v'$ in $\{A,\,B\}$ put
\[
 \mathsf M_{vv'}(d)=\frac14\sum_{r=0}^6
       \mathsf Q(v_r)^*\mathsf Q(v'_{r+d}).
\]
These operators lie in the algebra spanned by $\id$, $K_1$, $K_2$, $K_4$;
scalars below denote scalar operators. If $\mathsf M_{vv'}(d)=m_0+\sum_{a\in\mathcal Q}m_aK_a$, commutation
with $J_s$ and the identities $J_1w_0=iw_2$, $J_2w_0=iw_4$,
$J_3w_0=iw_1$ give
\[
 \frac14\left(\sum_r h(v_r,\,J_sv'_{r+d})\right)_{s=0}^3
       =(m_0,\,im_2,\,im_4,\,im_1).
\]
Indeed each component equals $h(w_0,\,J_s\mathsf M_{vv'}(d)w_0)$.
The Hermitian blocks $\mathsf Q_r=\mathsf Q(B_r)$ and
Lemma~\ref{lem:differencecorrelation} give
$\mathsf M_{BB}(d)=4\delta_{d,\,0}$, and hence
\begin{equation}\label{eq:BBcorrelation}
 \sum_r h(B_r,\,J_sB_{r+d})=16\delta_{d,\,0}\delta_{s,\,0}.
\end{equation}
For the other two cases, the seed formulas give
\[
\begin{aligned}
 \mathsf Q(A_0)&=-i+K_1+K_2+K_4,\qquad \mathsf Q(A_a)=0,\\
 \mathsf Q(A_{-a})&=\id+iK_a-K_{2a}+K_{4a}\qquad(a\in\mathcal Q).
\end{aligned}
\]
Multiplication using \eqref{eq:Qblock} and \eqref{eq:Pauliblocks} yields
\[
\begin{aligned}
 \mathsf M_{AA}(0)&=4+K_1+K_2+K_4,&
 \mathsf M_{AB}(0)&=i(\id+K_1+K_2+K_4),\\
 \mathsf M_{AA}(1)&=iK_1,&
 \mathsf M_{AB}(1)&=K_2-iK_4,\\
 \mathsf M_{AA}(3)&=-iK_4,&
 \mathsf M_{AB}(3)&=\id-iK_2.
\end{aligned}
\]
For example, $\mathsf Q(A_0)^*\mathsf Q(A_0)=4$ and
$\mathsf Q(A_{-a})^*\mathsf Q(A_{-a})=4(\id+K_{4a})$ give the first equality.
The order-three seed symmetry conjugates $K_a$ to $K_{2a}$ and sends
$d$ to $2d$. Thus these representatives of
$\{0\}$, $\{1,\,2,\,4\}$, $\{3,\,6,\,5\}$ give Gaussian-integer coefficients for
all shifts. The displayed matrix coefficients and their reversed
pairings therefore show that all signed seed translates pair to
multiples of four, since $\mathcal J$ is a group commuting with $P$.
The generators $2\cC$ pair trivially modulo four with these translates,
whose residues lie in $\cD$, and with one another. This proves
self-orthogonality.
The seed translates have squared norm sixteen; a representative $2u$,
$u\in\cC$, has squared norm $4\wt(u)\in8\Z$. Twice every cross-pairing
is divisible by eight. Thus every integer combination has squared norm
in $8\Z$, unchanged by reduction modulo four, proving evenness.
\end{proof}

\paragraph{Generation of the binary layer.}
The two Hamming layers give a direct generation argument over $\F_2$.

\begin{lemma}\label{lem:seedresidues}
The residues of the signed cyclic translates of $A$, $B$ span $\cD$.
Consequently $C$ is self-dual and even, with residue $\cD$, torsion
$\cC$, and additive structure $\Z_4^{19}\oplus\Z_2^{18}$.
\end{lemma}
\begin{proof}
The permutation parts of $I$, $J_2$, $J_1$ are the translations
$(p,\,e)\mapsto(p,\,e+\tau p)$ with $\tau=\id$, $T_P$, $T_P^4$.
Since $T_P^3=T_P+\id$, these span $\F_2[T_P]$.
The polynomial $t^3+t+1$ is irreducible, so every nonzero member
of this three-dimensional algebra is invertible. Consequently, for
$\ell\ne0$, the forms $\ell\circ\tau$ run through $E^*$.

The seven cyclic translates of $A\bmod2$ have linear coefficients
$\ell(p)\ell(e)$ for the seven nonzero $\ell\in E^*$. Their span is the
six-dimensional space of symmetric bilinear forms. Translation differences
are the constants $\ell(p)\ell(\tau p)$; these span all six quadratic
functions on $\Pi$ with zero constant term. The sum of the seven cyclic
residues is $\mathbf1_{56}$ by \eqref{eq:traceseedresidues}, giving all
seven block constants.

The translates of $B$ supply the three quadratic classes. Their
translation differences with $\tau=T_P$ give, modulo constants,
the forms $\det(c,\,T_Pp,\,e)$ for $c\in E$.
These form a three-dimensional complement to the symmetric forms.
Indeed, if $\det(c,\,T_Pp,\,e)$ is symmetric in $p$, $e$, then $T_P$ is
self-adjoint for the alternating form $\det(c,\,p,\,e)$ and preserves
its radical. For $c\ne0$ that radical is $\langle c\rangle$, whereas
$T_P$ has no invariant line. Thus $c=0$, proving the complement assertion.

We have all seven constants, nine bilinear coefficients, and three
quadratic classes. The residues therefore span $\cD$. Self-orthogonality
bounds the torsion by $\cC$, and the defining presentation includes
$2\cC$. Hence $|C|=2^{19+37}=2^{56}$, proving self-duality.
Evenness and the additive structure follow from
Lemmas~\ref{lem:codeorthogonal} and~\ref{lem:binaryduality}.
\end{proof}
Since $\GT$ permutes the seed translates, this lemma shows that it
preserves $\cD$, hence $\cC$ and the full code $C$.

\subsection{The signed maps used in the lattice geometry}\label{sec:H0}
The following signed maps determine the full code group and supply the
elements used for the second frame, local moves, and simplicity proofs.
Their preservation of $C$ is proved directly from the defining seeds.

The following tables use the complex coordinate basis
\eqref{eq:complexbasis}. They specify the images of $w_0$, $w_2$, $w_4$, $w_1$;
the images of their partners follow by commuting with $I$, so each entry
describes a signed permutation of the eight real positions.

\paragraph{A noncyclic row symmetry and the binary image.}
We need one further, noncyclic row permutation. Table~\ref{tab:R0} specifies a
complex monomial involution $R_0$. The four entries are the images of the
source-row basis $(w_0,\,w_2,\,w_4,\,w_1)$ in the target row.
\begin{table}[htbp]\centering\small
\begin{tabular}{cc l}\toprule
source row&target row&images in the target row\\\midrule
0&0&$(w_0,\,w_2,\,w_1,\,w_4)$\\
1&2&$(w_0,\,iw_4,\,-iw_1,\,w_2)$\\
2&1&$(w_0,\,w_1,\,-iw_2,\,iw_4)$\\
3&6&$(iw_0,\,-iw_4,\,w_2,\,w_1)$\\
4&4&$(w_4,\,-w_1,\,w_0,\,-w_2)$\\
5&5&$(-w_1,\,w_2,\,w_4,\,-w_0)$\\
6&3&$(-iw_0,\,w_4,\,iw_2,\,w_1)$\\\bottomrule
\end{tabular}
\caption{An explicit lift of a transvection of the seven row labels.}
\label{tab:R0}
\end{table}
In the binary labels of Table~\ref{tab:hamminglabels}, the operators $P$, $R_0$
have respective linear parts
\[
 T_P:\ (1,\,2,\,4)\longmapsto(5,\,3,\,2),\qquad
 T_R:\ (1,\,2,\,4)\longmapsto(5,\,2,\,4).
\]
Here and below a triple gives the columns of a binary matrix.
More explicitly, $P(p,\,e)=(T_Pp,\,T_Pe)$ and
\[
 R_0(p,\,e)=(T_Rp,\,T_Re+\delta_{T_Rp}),\qquad
 (\delta_{p_0},\,\ldots,\,\delta_{p_6})=(0,\,0,\,0,\,7,\,6,\,4,\,3).
\]
These formulas are identities on the eight named coordinates in each row.
They also verify preservation of $\cD$. For $g=T_R$, Table~\ref{tab:eta} gives
\begin{equation}\label{eq:Rsectionidentity}
\begin{aligned}
 &\theta_{gp}(gc)(ge+\delta_{gp})+\theta_p(c)(e)\\
 &\qquad=\theta_{gp}(gc)(\delta_{gp})
   +p_1\bigl(c_1(e_2+e_3)+(c_2+c_3)e_1\bigr).
\end{aligned}
\end{equation}
The right side belongs to $\cD_0$ for fixed $c$, and the affine map
preserves $\cD_0$. Thus $R_0$ preserves $\cD$ and its dual $\cC$.

\Needspace{6\baselineskip}
\begin{lemma}[The signed row involution]\label{lem:rowinvolution}
The displayed involution $R_0$ preserves $C$.
\end{lemma}
\begin{proof}
We first reduce the generators using the binary normal form. Put
\[
\begin{aligned}
 (u_1,\,\ldots,\,u_7)&=(A,\,PA,\,PJ_1A,\,J_1B,\,J_2B,\,PB,\,PJ_1B),\\
 (v_1,\,\ldots,\,v_8)&=(J_2A,\,P^2J_1A,\,P^2J_2A,\,P^3J_2A,\\
 &\hspace{34mm}P^5J_3A,\,P^6J_2A,\,B,\,P^4J_2B).
\end{aligned}
\]
Together with $2\cC$, these vectors and their $I$-multiples generate
$C$. To prove this, consider their binary residues.
Their $A$-terms have the six linear forms $\ell=1$, $2$, $3$, $5$, $6$, $7$ in
$\ell(p)\ell(e)$. The seven rank-one forms $\ell\ell^\top$ span the
six-dimensional symmetric space and sum to zero; hence any six are
independent. The differences $(1+I)$ applied to these terms give
all linear block constants. The three differences
\[
 (1+J_2)A,\qquad P^2(J_1+J_2)A,\qquad P(1+J_1)A
\]
have block-constant residues
$p_2+p_2p_3$, $p_2+p_1p_2$, $p_3+p_1p_3$, respectively.
They give all six even-weight block constants. The residue of
\[
 (1+I)B+P^3J_2A+P^5J_3A+P^6J_2A
\]
is the indicator of blocks $\{1,\,3,\,6\}$, supplying the seventh dimension.
The differences $(1+J_1)B$, $(1+J_2)B$, $P(1+J_1)B$ have bilinear matrices
with column triples $(1,\,7,\,6)$, $(7,\,7,\,1)$, $(4,\,2,\,0)$. Their $M+M^\top$
have their nonzero off-diagonal entries in positions $12$, $23$, $13$,
respectively, and hence complement the symmetric forms.
Finally $B$, $PB$, $P^4J_2B$ have quadratic classes $6$, $1$, $4$, a basis of $E$.
Thus all $7+9+3$ parameters of \eqref{eq:binarynormal} occur.
Equality of the quaternary codes follows because $2\cC$ is included.

The images of eight of these generators are seed translates
with Hamming block signs. The four-column rule in Table~\ref{tab:R0}
and \eqref{eq:seedarrays} give
\[
\begin{aligned}
 (R_0v_1,\,R_0v_2,\,R_0v_3,\,R_0v_4)
   &=(-E_3v_1,\,-E_2v_2,\,-E_3v_3,\,E_3v_4),\\
 (R_0v_5,\,R_0v_6,\,R_0v_7,\,R_0v_8)
   &=(-iE_6v_6,\,iE_7v_5,\,-E_3v_7,\,E_1v_8).
\end{aligned}
\]
They therefore lie in $C$. For the remaining seven generators, the
same formulas give the following Hermitian pairings:
\begin{equation}\label{eq:Rpairings}
 \frac14\bigl(h(u_j,\,R_0u_k)\bigr)_{j,\,k=1}^7=
 \begin{pmatrix}
 1&0&i&0&0&i&0\\
 0&1&0&0&1&-1&i\\
 -i&0&1&0&0&0&0\\
 0&0&0&0&1&0&0\\
 0&1&0&1&0&0&i\\
 -i&-1&0&0&0&1&i\\
 0&-i&0&0&-i&-i&1
 \end{pmatrix}.
\end{equation}
This matrix is Hermitian, so only its $28$ entries on or above the
diagonal need checking. Since $R_0$ is self-adjoint and commutes with $I$,
pairings of $R_0u_j$ with $v_k$, $Iv_k$ are divisible by four because
$R_0v_k\in C$ and $C$ is self-orthogonal. Equation~\eqref{eq:Rpairings}
handles the pairings with $u_k$, $Iu_k$.
Pairings with $2\cC$ are automatic because the binary permutation
preserves $\cD$. Self-duality now gives $R_0u_j\in C$.
Also $R_0(2\cC)=2\cC$, so the generating reduction proves $R_0C=C$.
\end{proof}

The matrices $T_R$, $T_PT_RT_P^{-1}$, $T_P^4T_RT_P^{-4}$ are, respectively,
$\id+E_{31}$, $\id+E_{23}$, $\id+E_{12}$. Their commutators give the other elementary
transvections, so they generate $\operatorname{GL}_3(2)$, of order
$(8-1)(8-2)(8-4)=168$.

\paragraph{Three derived monomial involutions.}
Define
\begin{equation}\label{eq:Uword}
 U=-E_1\,\GT R_0\GT^{-1}R_0\GT.
\end{equation}
Its action on the four basis vectors in each row is the $U$ column
of Table~\ref{tab:UVW}. Put $U_m=P^mUP^{-m}$, and let
$E_5\in E_C$ be the row sign change $(-1)^{5\cdot p}$, negating rows
$0$, $1$, $4$, $6$. Set
\begin{equation}\label{eq:VWwords}
\begin{aligned}
 V&=J_1J_2 U_0U_4U_3U_1U_2,\\
 W&=E_5J_2 U_1U_2U_3.
\end{aligned}
\end{equation}
Multiplication gives the $V$, $W$ columns in the same table.
All three maps preserve $C$ as products of known automorphisms;
the table also shows that they are block-preserving involutions.
The entries in a row give the images of the local basis, not the output
coefficients of a vector.
\begin{table}[htbp]\centering\small
\setlength{\tabcolsep}{4pt}
\begin{tabular}{c@{\quad}lll}\toprule
$r$&$U(w_0,\,w_2,\,w_4,\,w_1)$&$V(w_0,\,w_2,\,w_4,\,w_1)$&$W(w_0,\,w_2,\,w_4,\,w_1)$\\\midrule
0&$(-w_0,\,-w_2,\,-w_4,\,-w_1)$&$(-w_0,\,-w_2,\,w_4,\,w_1)$&$(w_0,\,-w_2,\,w_4,\,-w_1)$\\
1&$(w_1,\,iw_4,\,-iw_2,\,w_0)$&$(w_4,\,iw_1,\,w_0,\,-iw_2)$&$(w_0,\,w_2,\,w_4,\,w_1)$\\
2&$(w_2,\,w_0,\,-iw_1,\,iw_4)$&$(-iw_1,\,w_4,\,w_2,\,iw_0)$&$(w_1,\,-iw_4,\,iw_2,\,w_0)$\\
3&$(iw_2,\,-iw_0,\,w_1,\,w_4)$&$(iw_4,\,w_1,\,-iw_0,\,w_2)$&$(w_0,\,w_2,\,-w_4,\,-w_1)$\\
4&$(-w_0,\,w_2,\,-w_4,\,w_1)$&$(w_0,\,w_2,\,w_4,\,w_1)$&$(iw_2,\,-iw_0,\,-w_1,\,-w_4)$\\
5&$(w_0,\,-w_2,\,-w_4,\,w_1)$&$(w_0,\,-w_2,\,-w_4,\,w_1)$&$(-w_4,\,iw_1,\,-w_0,\,-iw_2)$\\
6&$(-iw_4,\,w_1,\,iw_0,\,w_2)$&$(w_2,\,w_0,\,iw_1,\,-iw_4)$&$(iw_1,\,w_4,\,w_2,\,-iw_0)$\\\bottomrule
\end{tabular}
\caption{The derived monomial involutions $U$, $V$, $W$.}
\label{tab:UVW}
\end{table}

The coordinate permutations of $I$, $U$, $V$, $W$ are
$t_\tau:(p,\,e)\mapsto(p,\,e+\tau p)$, with
\begin{equation}\label{eq:nineA}
 \tau_I=\id,\qquad \tau_U=(6,\,0,\,0),\qquad
 \tau_V=(4,\,0,\,0),\qquad \tau_W=(0,\,1,\,1).
\end{equation}
The triples specify columns. In particular $\tau_U$ has rank one and
square zero. Its conjugates and $\tau_I$ generate all translations,
by the following elementary observation.

\begin{lemma}[The translation span]\label{lem:translationspan}
The $\operatorname{GL}(E)$-conjugates of any nonzero rank-one
square-zero endomorphism span the trace-zero endomorphisms. This
subspace is also spanned by all conjugation differences
$a\tau a^{-1}-\tau$, and is complementary to $\langle\id\rangle$.
\end{lemma}
\begin{proof}
The conjugates are exactly the maps $u\ell$ with $u\ne0$, $\ell\ne0$ and
$\ell(u)=0$. They include the off-diagonal matrix units $E_{ij}$ and,
for $i\ne j$, the rank-one square-zero matrix
$E_{ii}+E_{ij}+E_{ji}+E_{jj}$. Adding $E_{ij}+E_{ji}$ to the latter
also gives every trace-zero diagonal direction. All these maps have
trace zero, proving the first assertion.

The span of conjugation differences is conjugation-stable and has trace zero.
It contains $E_{ij}$, since
$(\id+E_{ij})E_{jj}(\id+E_{ij})^{-1}-E_{jj}=E_{ij}$.
The first assertion therefore identifies it with the trace-zero subspace.
Finally $\operatorname{tr}(\id)=1$ in dimension three over $\F_2$,
which gives the stated complement.
\end{proof}

\begin{proposition}[The full signed code group]\label{prop:H0}
The full signed automorphism group is
\[
 H=\Aut(C)=\langle E_C,\,\mathcal J,\,P,\,\GT,\,R_0\rangle,
 \qquad |H|=2^{13}\cdot168.
\]
Every element commutes with $I$. Its block kernel $P_0$ has order $2^{13}$,
and the sequences
\begin{equation}\label{eq:codeautexact}
\begin{aligned}
 1&\longrightarrow E_C\longrightarrow H\longrightarrow\Gamma\longrightarrow1,\\
 1&\longrightarrow E_C\longrightarrow P_0
        \longrightarrow\operatorname{End}(E)\longrightarrow1
\end{aligned}
\end{equation}
are exact. Here $E_C\cong\cH_7$ centralizes $P_0$, and every block
translation has exactly sixteen signed lifts.
\end{proposition}
\begin{proof}
Every displayed generator preserves $C$, by the defining presentation,
the sign-kernel theorem, and the preceding signed-map construction.

Let $K$ be the subgroup generated by the displayed maps; it contains
$U$ by \eqref{eq:Uword}. Its binary image maps onto
$\operatorname{GL}(E)$, by the transvections obtained from $P$, $R_0$.
Conjugating $t_\tau$ by an affine map with common linear part $g$ on $p$, $e$
gives $t_{g\tau g^{-1}}$; the affine offsets cancel. The image of $K$
therefore contains $t_\tau$ for every conjugate of $\tau_U$.
By Lemma~\ref{lem:translationspan}, these translations span the
trace-zero endomorphisms. The translation $\tau_I=\id$ supplies the
remaining direction. Thus the binary image of $K$ is all of
$\Gamma$.

Every signed automorphism has binary image in $\Gamma$, and its sign
kernel is exactly $E_C$ by Theorem~\ref{thm:signkernel}. As $K$ contains
this kernel, $K=\Aut(C)$ and
$|H|=16|\Gamma|=2^{13}\cdot168$. Taking the inverse image of the
translation subgroup gives $P_0$ and both exact sequences.
Block-constant signs commute with block-preserving signed permutations,
so $E_C$ centralizes $P_0$. Finally every displayed generator commutes
with $I$, either by \eqref{eq:signedgenerators} or by its complex monomial
definition, proving the assertion for $H$.
\end{proof}
The scalar group is included in this count: $I$ induces the translation
$\tau=\id$.

\subsection{The words of least Euclidean weight}\label{sec:lifts}
The following block-parity functionals distinguish the residue supports
that admit minimum words.

\paragraph{A block lift and its norm bound.}
The following explicit lift of a block indicator gives the norm bound
and the minimum-support obstructions needed below. The array uses
the signed basis of Table~\ref{tab:hamminglabels}.
Consider the integer generator combination
\[
 d_*=J_1 A+J_1 PA+ I P^2A+P^3A+P^4A+P^5A+P^6A.
\]
Its reduction modulo four is
\begin{equation}\label{eq:blocklift}
\begin{pmatrix}
1&1&1&1&1&1&1&1\\
2&2&2&0&0&2&2&0\\
0&0&0&2&0&2&0&2\\
2&0&0&2&0&0&2&0\\
2&0&0&2&0&0&2&0\\
0&0&0&2&2&0&2&0\\
0&2&2&2&2&2&0&0
\end{pmatrix}.
\end{equation}
Let $d_r=P^rd_*\bmod4$. Its binary residue is the indicator of row $r$;
its coordinate sum in every other row is $2\bmod4$.

\begin{lemma}[Block parity and the norm bound]\label{lem:blocknormbound}
For $p=p_r$, write $d_r=\mathbf1_p+2w^{(p)}$, where $\mathbf1_p$ is the block indicator
and $w^{(p)}\in\F_2^{56}$. The functional $\phi_p(d)=w^{(p)}\cdot d$ on $\cD$
is independent of the chosen lift of $\mathbf1_p$. If $u\in C$ has residue $d$
vanishing on block $p$, then
\begin{equation}\label{eq:blockparity}
 \#\{e:u(p,\,e)=2\}\equiv\phi_p(d)\pmod2.
\end{equation}
In particular, every lift $u\in C$ of a block indicator satisfies
\begin{equation}\label{eq:blocknormbound}
 \operatorname{wt}_E(u)\ge8+6\cdot4=32.
\end{equation}
\end{lemma}
\begin{proof}
Changing the lift changes $w^{(p)}$ by a word of $\cC=\cD^\perp$, so
$\phi_p$ is well defined. Self-orthogonality gives
$0=d_r\cdot u=\sum_eu(p,\,e)+2w^{(p)}\cdot d\pmod4$, proving
\eqref{eq:blockparity}. For $p\ne q$, the row sums in
\eqref{eq:blocklift} give $\phi_p(\mathbf1_q)=1$. A lift of $\mathbf1_q$ therefore
has eight odd entries in block $q$ and at least one two in each of the
other six blocks. This proves the bound.
\end{proof}

For a four-block residue
$d(p,\,e)=\epsilon_p+p_1\ell'(e)$, with $\epsilon_p=0$ when $p_1=0$,
put $\varepsilon=\sum_{p:p_1=1}\epsilon_p$.
For the inactive labels $q\in\{2,\,4,\,6\}$, the same block lift gives
\begin{equation}\label{eq:obstructiontable}
\begin{aligned}
 \gamma_q&=\sum_{p:p_1=1}\sum_{e\in E}w^{(q)}(p,\,e)e,
       & (\gamma_2,\,\gamma_4,\,\gamma_6)&=(5,\,3,\,1),\\
 \phi_q(d)&=\varepsilon+\ell'(\gamma_q).
\end{aligned}
\end{equation}
Indeed every active-row sum of $w^{(q)}$ is odd; its first moments are
obtained from \eqref{eq:blocklift} using Table~\ref{tab:hamminglabels}.
In these vector labels $\gamma_q=q+7$, so
$\gamma_p+\gamma_q=p+q$ for distinct inactive labels $p$, $q$.

\begin{lemma}\label{lem:lowexclusions}
Let $z\in\Z^{56}$ satisfy $z\bmod4\in C$ and $\sum z_j^2=16$.
If $z\bmod2\ne0$, then $z$ has precisely sixteen nonzero entries, all equal
to $\pm1$, and their support is not the union of two complete rows.
\end{lemma}
\begin{proof}
The block-lift bound \eqref{eq:blocknormbound} excludes a residue of
weight eight. A nonzero residue therefore has weight sixteen by
Theorem~\ref{thm:WD}; the norm leaves exactly sixteen entries $\pm1$.

For two complete rows, powers of $P$ reduce their labels to distinct
$p$, $q$ in $\{2,\,4,\,6\}$, since $P$ cycles the seven Fano lines.
Choose $\ell'\in E^*$ with $\ell'(p+q)=1$ and take any lift $u\in C$
of the residue $d(r,\,e)=r_1\ell'(e)\in\cD_0$.
This residue vanishes on both rows, and
\eqref{eq:blockparity}--\eqref{eq:obstructiontable} give
\[
 \#\{(r,\,e):r\in\{p,\,q\},\ u(r,\,e)=2\}
 \equiv\phi_p(d)+\phi_q(d)=\ell'(p+q)=1\pmod2.
\]
Hence $u\cdot z=2\pmod4$ for every signed word on these two rows,
contrary to self-orthogonality. This excludes the two-row support.
\end{proof}

\begin{lemma}\label{lem:foursupports}
Among the $784$ four-block supports, write $M=\ell'\ell^{\top}$ with $\ell$, $\ell'$ in $E^*\setminus\{0\}$.
A signed lift of Euclidean weight $16$ is possible only when
\[
 \ell'=\ell,\qquad \sum_{p:\ell(p)=1}\epsilon_p=1.
\]
There are exactly $56$ such supports, each with exactly $32$ signed lifts.
For each fixed label $\ell$, the row kernel $P_0$ is transitive on
its eight supports and all their signed lifts.
\end{lemma}
\begin{proof}
The cyclic row permutation is transitive on the seven nonzero forms $\ell$.
We may therefore assume $\ell=1$.
A lift of Euclidean weight $16$ has no twos, so all three functionals
in \eqref{eq:obstructiontable} vanish. Subtracting their equations gives
$\ell'(6)=\ell'(2)=0$.
Since $\ell'\ne0$, this forces $\ell'=1$, and then $\varepsilon=1$.
This proves necessity and the count $7\cdot8=56$.

The seed $A$ has one such support. The translations in $P_0$ change the
active constants by $\ell(\tau p)$, whose restrictions are all eight even-weight
functions on the affine plane $\ell(p)=1$. They are therefore transitive on the
eight odd choices. The cyclic row map connects the seven labels $\ell$. Every
admissible support consequently has at least one signed lift.

It remains to count the fibre and prove sign transitivity. For an
admissible support $\Sigma$, put
\[
 \cC^0_\Sigma=\cD_0^\perp\cap\F_2^\Sigma,\qquad
 \cC_\Sigma=\cD^\perp\cap\F_2^\Sigma.
\]
The selected plane in an active row is $\ell(e)=1+\epsilon_p$.
A word $v$ supported in $\Sigma$ annihilates $\cD_0$ exactly when it
selects an even number of positions in each active row and its moments
\[
 s_p=\sum_{e\in\Sigma_p}v(p,\,e)e,\qquad
 \Sigma_p=\{e:(p,\,e)\in\Sigma\},
\]
are a common vector $s\in \ker\ell$. Indeed
$\sum_p p\otimes s_p=0$, and the four active row labels have just one
linear relation, their sum. An even subset of a four-point affine plane
with prescribed first moment has two choices, differing by the whole
plane. Hence $\dim\cC^0_\Sigma=2+4=6$.
For $c\in\ker\ell$, the restriction of $\theta_p(c)$ to each selected
plane is affine. These restrictions also satisfy the compatibility
between blocks. Indeed, reduce to $\ell=1$. Table~\ref{tab:eta} gives
\[
 \sum_{p:p_1=1}\eta_2(p)=5,\qquad
 \sum_{p:p_1=1}\eta_4(p)=3.
\]
Since $\sum_{p:p_1=1}\epsilon_p=1$, the four affine slopes sum to zero:
\[
 \sum_{p:p_1=1}\bigl(\theta_p(c)(s)
       +\det(c,\,(1+\epsilon_p)1,\,s)\bigr)=0
       \qquad(c,\,s\in\ker\ell).
\]
Here the binary label $1$ denotes $(1,\,0,\,0)\in E$. For $c=2$ the two
sums are both $s_3$, and for $c=4$ both are $s_2$; linearity in $c$
gives the stated cancellation.
The four active row labels have just their sum as a relation, so these
slopes extend linearly in $p$. Hence the restrictions for $c\in\ker\ell$
belong to $\cD_0|_\Sigma$.
A class with $\ell(c)=1$ contributes one new dimension: the indicator
of a selected plane annihilates $\cD_0$ but pairs with $\theta_p(c)$
as $\ell(c)$, by the polar form. Since
$\dim(\cD_0|_\Sigma)=16-6=10$, this proves
\begin{equation}\label{eq:liftfiber}
 \dim(\cD|_\Sigma)=11,\qquad
 \dim\cC_\Sigma=5,\qquad
 \#\{\text{signed lifts on }\Sigma\}=2^5=32.
\end{equation}

For sign transitivity, reduce to $\Sigma=\operatorname{supp}(A)$,
with label $\ell=2$. The restrictions of $E_C$ give all three dimensions
of even block-constant patterns; their kernel consists of the indicator
of the three inactive rows and zero. The common-moment map on
$\cC_\Sigma$ therefore has this three-dimensional kernel and is onto
$\ker\ell$. The support-fixing maps $U_1$, $U_2$ give sign differences
$(A-U_1A)/2$, $(A-U_2A)/2$ with common moments $1$, $5$, a basis of
$\ker\ell$, by Table~\ref{tab:UVW}.
A support-fixing translation leaves torsion moments unchanged because
each row has even weight, so these moment shifts add under composition.
They reach all four moment classes, and $E_C$ is transitive on each
kernel fibre. All these maps lie in $P_0$; the cyclic row symmetry
transports the assertion to every label.
\end{proof}

\begin{lemma}\label{lem:heavysupports}
Each of the $896$ heavy supports has exactly sixteen signed lifts.
For each fixed heavy label $c$, $P_0$ is transitive on its $128$ supports
and on all $2048$ vectors on them.
\end{lemma}
\begin{proof}
In \eqref{eq:binarynormal}, heaviness means $M^{\top}c=0$.
Translation by $\tau p$ changes $M$ by $[c\times]\tau$, giving all $64$
matrices with that property. The remaining translations $\tau=c\ell^{\top}$
change only the heavy-block constant, by $\ell(c)$, so both constants
occur. The seed $B$ supplies one signed lift, and powers of $P$ give
one in every heavy family. Since all the translations have signed lifts
in $P_0$, every heavy support occurs.

Fix such a support $\Sigma$. In a light block its pair has difference $c$.
In the heavy block its four points have sum $c$ and contain one point
of each pair modulo $\langle c\rangle$, by \eqref{eq:balancedidentity}.
A word of $\cC_\Sigma$ selects an even subset in each block. Let $s_c$
be its first moment in the heavy block, and let $\epsilon_p$ record whether
it selects the whole pair in a light block $p$. Orthogonality to bilinear
words gives
\[
 c\otimes s_c+\sum_{p\ne c}\epsilon_p p\otimes c=0.
\]
It forces $s_c=0$ or $c$. No pair in the heavy block has difference $c$,
so that block, like every light block, is empty or full. The equation
reduces to $\sum_p\epsilon_pp=0$. Conversely all sixteen Hamming block
signs restrict to words of $\cC_\Sigma$, by Theorem~\ref{thm:signkernel}.
Their restrictions are distinct since every block occurs. Thus the
sixteen signed lifts form a free transitive $E_C$-orbit. Together with
support transitivity this proves the assertion.
\end{proof}

\begin{theorem}\label{thm:codewords16}
The minimum Euclidean weight of $C$ is $16$. Its words of that
weight have sixteen entries $\pm1$ and all remaining entries zero. They form
two $H$-orbits, consisting of $56\cdot32=1792$ four-block words and
$896\cdot16=14336$ heavy-block words.
\end{theorem}
\begin{proof}
By evenness, a positive weight below sixteen would be eight.
Its residue would be zero or a block indicator, by Theorem~\ref{thm:WD}.
The first case would require a weight-two word of $\cC$, excluded by
Proposition~\ref{prop:torsionmin}; the second contradicts
\eqref{eq:blocknormbound}. Since $A$ has weight sixteen, this is the minimum.
A word of weight $16$ with zero residue would give a weight-four word
of $\cC$, also excluded. Lemmas~\ref{lem:lowexclusions}--\ref{lem:heavysupports}
therefore exhaust the possibilities and give the stated counts. The heavy
words form one $H$-orbit by Lemma~\ref{lem:heavysupports} and the cyclic row map.
The four-block assertion follows from
Lemma~\ref{lem:foursupports} and the cyclic row symmetry.
\end{proof}

\section{The Gaussian lattice and its integral reflection}\label{sec:construction}
Let $C$ be the self-dual even quaternary code of \eqref{eq:seedcode}.
We define
\begin{equation}\label{eq:lattice}
 L=\tfrac12\{z\in\Z^{56}:z\bmod4\in C\}.
\end{equation}
This section equips $L$ with its Gaussian structure, identifies its
coordinate-frame stabilizer, and constructs an additional integral
reflection from a second orthogonal frame. The frame exchange is then
recovered from two reflections.

\subsection{The lattice and its coordinate frame}
On the real coordinate space $\R^{56}$, the signed permutation $I$ is
orthogonal and satisfies $I^2=-\id$. Regard it as multiplication by $i$.
Use the Hermitian form $h$ of \eqref{eq:hermitianform}.
Then $2\Z^{56}\subseteq L\subseteq\tfrac12\Z^{56}$ and $iL=L$.
The defining presentation consists of the coordinate vectors
$2\mathbf e_{r,\,j}$, the halves of the signed seed translates in
\eqref{eq:seedcode}, and the integer lifts of $\cC$.

\begin{corollary}\label{cor:unimod}
The underlying real lattice is unimodular even of rank $56$. The lattice
$L$ is Gaussian integral and has minimum four.
\end{corollary}
\begin{proof}
Self-orthogonality of $C$ gives integrality of $L$, and
$\operatorname{wt}_E(C)\subseteq8\Z$ gives evenness. The sublattice $2\Z^{56}$ has covolume $2^{56}$,
and its index in $L$ is $|C|=2^{56}$, so $L$ is unimodular.
Since $I L=L$, integrality of both $(v,\,w)$ and $( I v,\,w)$ proves
Gaussian integrality. A nonzero doubled vector of squared norm less than
$16$ would either lie in $4\Z^{56}$, where this is impossible, or represent
a nonzero codeword of Euclidean weight less than $16$. Theorem~\ref{thm:codewords16}
excludes the latter possibility. The coordinate vectors have squared norm
four, so the minimum is four.
\end{proof}

A \emph{cross} is $X(v)=\mu_4v$ for a minimal vector $v$, where
$\mu_4=\{1,\,-1,\,i,\,-i\}$. Write $\cX$ for the set of crosses. A frame consists of $28$
mutually Hermitian-orthogonal crosses. Set
\[
 (f_{4r+1},\,\ldots,\,f_{4r+4})=2e_r(w_0,\,w_2,\,w_4,\,w_1)
       \quad(0\le r\le6),\qquad
 X_j=X(f_j),
\]
where $e_rv$ means that $v$ occupies row $r$, and $1\le j\le28$.
Thus each consecutive group of four indices is one coordinate row;
this single numbering is used for vectors and crosses throughout.
The coordinate frame and the fixed local data are
\[
 F=\{X_1,\,\ldots,\,X_{28}\},\qquad
 x=X_1,\quad y=X_2,\quad Q_1=\{X_1,\,X_2,\,X_3,\,X_4\}.
\]
Subscripts on stabilizers indicate that listed crosses are fixed
individually; a quartet or frame appearing as a subscript is fixed setwise.

\begin{corollary}\label{cor:framestabilizer}
The code group $H$ is the full Gaussian stabilizer of $F$.
\end{corollary}
\begin{proof}
By Proposition~\ref{prop:H0}, all signed automorphisms of $C$
commute with $I$. They therefore preserve $L$, $h$ and $F$.
Conversely, a Gaussian isometry preserving $F$ permutes the $56$ real
coordinate directions up to signs. Reducing $2L$ modulo four shows that
it preserves $C$. It consequently belongs to $H$.
\end{proof}
The element $\GT\in H$ fixes $f=f_1$, so $\GT\in H_x$.

\begin{theorem}\label{thm:shellproved}
The minimal shell has $112+1792+14336=16240$ vectors and
$28+448+3584=4060$ Gaussian crosses. The group $H$ is transitive on each
of these three families.
\end{theorem}
\begin{proof}
A minimal vector whose doubled array reduces to zero modulo four is one
of the $112$ coordinate vectors. Every other doubled array has squared norm
$16$, so its codeword has Euclidean weight $16$ and is one of the two families
in Theorem~\ref{thm:codewords16}. Its entries are $\pm1$, so each such
codeword has a unique lift of squared norm $16$. This gives the vector counts.
The two codeword families are $H$-orbits. The coordinate vectors also form
one orbit, since $\mathcal J$ is transitive on the sixteen signed coordinate
vectors of a block and $P$ cycles the seven blocks. Finally $\mu_4$ acts
freely on nonzero vectors, giving the cross counts and transitivity.
\end{proof}

\begin{remark}[Comparison with the classical lattice]\label{rem:identify}
In the displayed signed coordinates, $L$ is the classical Conway--Wales
lattice normalized to minimum four; the seeds are those reproduced in
\cite[\S5.1, Tables~1--2]{DuncanII} from \cite{Conway}.
This comparison is separate from the construction, automorphism-group
order, and simplicity proofs below, none of which uses a classical
lattice or group identification.
\end{remark}

\subsection{The second orthogonal frame}\label{sec:S}
Put $f=f_1$ and $b=B/2$. Define $F'$ as the orbit of the heavy seed
cross under $\langle P,\,\mathcal J\rangle$, with the labeling
\begin{equation}\label{eq:secondframe}
 F'=\langle P,\,\mathcal J\rangle\cdot X(b)
   =\{X(P^{-r}J_sb):0\le r\le6,\ 0\le s\le3\}.
\end{equation}
Its $56$ real representatives, used for the binary glue, are indexed
in the original eight real positions:
\begin{equation}\label{eq:secondrealframe}
 (b_{r,\,0},\,\ldots,\,b_{r,\,7})
 =P^{-r}(-ib,\,b,\,-iJ_1 b,\,J_1 b,\,-iJ_2 b,\,J_2 b,\,-iJ_3b,\,J_3b).
\end{equation}

\begin{lemma}[The labeled second frame]\label{lem:secondframe}
The set $F'$ is a frame of heavy minimal crosses, disjoint from $F$.
The vectors $b_{r,\,j}$ form a real orthogonal basis of squared norm four
and belong to $L$. The operators $\mathcal J$ and $\GT$ permute them
with the same signed index actions as on the coordinate representatives
$2\mathbf e_{r,\,j}$.
\end{lemma}
\begin{proof}
Every displayed vector is an $H$-image of $b$, so it belongs to $L$ and
is a heavy minimum by Theorem~\ref{thm:shellproved}. To compare two
members of \eqref{eq:secondframe}, move their signed quaternion operators
to one side of the Hermitian pairing. Since $\mathcal J$ is a group and
commutes with $P$, the result, up to a Gaussian unit, is one of the
$BB$ correlations in \eqref{eq:BBcorrelation}. It is zero for distinct
indices and four on the diagonal. Thus the $28$ crosses are mutually
orthogonal. Taking a representative and its $i$-multiple on each cross
gives the $56$ real orthogonal vectors in \eqref{eq:secondrealframe}.
None of their crosses is coordinate, because the heavy family is disjoint
from the coordinate family.

For the signed index actions, replacing $b$ by $f$ and $P^{-r}$ by $P^r$
in \eqref{eq:secondrealframe} gives the coordinate representatives
$2\mathbf e_{r,\,j}$. Left multiplication by a member of $\mathcal J$
permutes the eight listed operators up to signs, independently of whether
they are applied to $f$ or to $b$. This proves the assertion for
$\mathcal J$. By Lemma~\ref{lem:frobenius}, $\GT$ fixes both $f$ and
$b$, normalizes $\mathcal J$, and conjugates $P$ to $P^2$. It therefore
sends $r$ to $2r$ and has the same signed action on the inner index in
both frames. This proves the assertion for $\GT$.
\end{proof}

We use the same seven-block labels for $F'$ as for $F$.
For a Fano line $\lambda\subset\Pi$ write
$\mathcal R_\lambda=\{r:p_r\in\lambda\}$. Its $24$ real frame vectors span a
complex $12$-space. For the line chosen below, the binary glue embeds a
sublattice whose dual is half itself; the reflection in its span therefore
preserves $L$. The resulting frame-code identification is given in
Remark~\ref{cor:secondframecode}.

\subsection{A self-dual three-block code}\label{sec:threeblockcode}
For a Fano line $\lambda$ of block labels abbreviate
\[
 \cC_\lambda=\cC_{\lambda\times E},
\]
identifying words supported on $\lambda\times E$ with their restrictions
to these $24$ positions.
\begin{lemma}\label{lem:threeblockselfdual}
The shortened binary code $\cC_\lambda$ is self-dual of length $24$
and dimension $12$, and
\[
 \cC_\lambda=\cD|_{\lambda\times E}.
\]

\end{lemma}
\begin{proof}
The restriction $\cD|_{\lambda\times E}$ has three block constants,
six independent linear coefficients, and three common quadratic parameters. It is therefore
$12$-dimensional. The products of affine functions sum to zero on each
block. The other pairings vanish on summing over $\lambda$, by
\eqref{eq:productidentity} and $\sum_{p\in\lambda}p=0$.
Thus the restriction is self-orthogonal and hence self-dual. Duality of
puncturing and shortening gives the asserted identification with
$\cC_{\lambda\times E}$. 
\end{proof}

We need only one orbit generator for this code. Choose the line
\[
 \lambda=\{1,\,4,\,5\}\subset\Pi,
 \quad\text{that is, coordinate rows }\{1,\,2,\,4\}.
\]
It is fixed by $\GT$. In the inner Hamming coordinates let $\mathbf q$ have support
\begin{equation}\label{eq:gluehexad}
 \{(1,\,5),\,(1,\,6),\,(5,\,0),\,(5,\,3),\,(4,\,1),\,(4,\,2)\}.
\end{equation}
In the original real basis $(\mathbf e_0,\,\ldots,\,\mathbf e_7)$ it selects
positions $5$, $6$ of row $1$, positions $1$, $3$ of row $2$, and positions
$3$, $5$ of row $4$. These are basis indices, not column numbers in the
signed ordering of Table~\ref{tab:hamminglabels}.

\begin{lemma}\label{lem:hexadorbit}
The permutation orbit of $\mathbf q$ under $\langle\mathcal J,\,\GT\rangle$
spans $\cC_\lambda$.
\end{lemma}
\begin{proof}
The word has common quadratic class $c=3$. In the section coordinates
its remaining constants, on block labels $1$, $5$, $4$, are $(0,\,1,\,0)$ and
its linear corrections are $(7,\,4,\,3)$; the latter sum to zero. Thus
$\mathbf q\in \cC_\lambda$. Frobenius sends its quadratic class through
$3$, $2$, $7$, which are a basis of $E$.

Let $\widehat\lambda=\lambda\cup\{0\}$. Translation differences using $I$, $J_2$ supply,
modulo block constants, the restrictions to $\widehat\lambda\times E$ of
\[
 \det(c,\,p,\,e)+\det(d,\,T_Pp,\,e),\qquad c,\,d\in E.
\]
These six parameters are independent on $\widehat\lambda$. Choose nonzero
$u\in \widehat\lambda\cap T_P^{-1}\widehat\lambda$ and put $v=T_Pu$.
Since $T_P$ has no invariant line or plane, $u$, $v$, $T_Pv$ are independent.
If the displayed form vanishes, its equation at $p=u$ forces
$c$, $d$ to lie in $\widehat\lambda=\langle u,\,v\rangle$. Its equation at $p=v$, compared in
this basis, then gives $d=0$ and $c\in\langle v\rangle$;
the first equation now gives $c=0$.
Thus the first differences give all six linear coefficient maps
$\widehat\lambda\to E^*$, modulo constants.

Apply the translation $e\mapsto e+p$ once more. These differences
give the block-constant functions $p\mapsto M(p)(p)$, with $M:\widehat\lambda\to E^*$
arbitrary. On a two-dimensional plane their restrictions to the three
nonzero points give all three independent functions (the two linear
monomials and their product suffice). Thus the generated code contains
all three constants, all six linear coefficients, and all three
quadratic classes. Its dimension is $12$, as required.
\end{proof}

\subsection{Integral reflections from self-dual binary glue}\label{sec:integralreflection}
The following lemma gives an integral reflection. The dual is taken in
the real Euclidean span of $\Lambda$.
\begin{lemma}\label{lem:modularreflection}
Let $L$ be integral and let $\Lambda\subseteq L$ be a sublattice with
$2\Lambda^*\subseteq \Lambda$. Then
\[
 r_\Lambda=2\pr_{\Lambda_{\R}}-\id
\]
preserves $L$. It is a Gaussian isometry when $\Lambda_{\R}$ is complex.
Primitivity of the embedding $\Lambda\subseteq L$ is not required.
\end{lemma}
\begin{proof}
For $v\in L$, integrality gives $\pr_{\Lambda_{\R}}v\in \Lambda^*$. Hence
$2\pr_{\Lambda_{\R}}v\in \Lambda\subseteq L$ and $r_\Lambda v\in L$. Since $r_\Lambda^2=\id$,
inclusion is equality. Projection onto a complex subspace commutes
with multiplication by $i$.
\end{proof}

\begin{proposition}\label{prop:secondframereflection}
For the $\GT$-stable line $\lambda=\{1,\,4,\,5\}$ of the second frame,
the real span of its $24$ frame vectors contains a sublattice
$\Lambda_\lambda\subseteq L$ with
$\Lambda_\lambda^*=\tfrac12\Lambda_\lambda$. Consequently the reflection fixing
these twelve complex directions and negating their complement is a
Gaussian lattice isometry.
\end{proposition}
\begin{proof}
The row indices of this line are $\mathcal R_\lambda=\{1,\,2,\,4\}=\mathcal Q$. Write $\cC_\lambda$ in the real position order of these
blocks. Let $F_\lambda$ be the integer span of the $24$ mutually
orthogonal vectors $b_{r,\,j}$, $r\in \mathcal R_\lambda$, of norm four, and set
\begin{equation}\label{eq:secondframemodule}
 \Lambda_\lambda=F_\lambda+
 \left\{\frac12\sum_{(r,\,j)\in \mathcal R_\lambda\times\{0,\,\ldots,\,7\}}
          c_{r,\,j}b_{r,\,j}:c\in \cC_\lambda\right\}.
\end{equation}
This is a lattice because addition of binary words introduces only
integer frame vectors. In the orthonormal basis $b_{r,\,j}/2$ it is
$2\Z^{24}+\cC_\lambda$. A vector pairs integrally with $2\Z^{24}$
exactly when it has the form $u/2$, with $u\in\Z^{24}$; pairing it
with the glue requires $u\bmod2\in \cC_\lambda^\perp=\cC_\lambda$.
Consequently $\Lambda_\lambda^*=\tfrac12\Lambda_\lambda$. This establishes the
duality property from the binary code; it remains to prove that the
glue vectors belong to $L$.

For this embedding a single six-vector identity suffices:
\begin{equation}\label{eq:hexadembedding}
 \frac{-b_{1,\,5}+b_{1,\,6}+b_{2,\,1}+b_{2,\,3}-b_{4,\,3}-b_{4,\,5}}{2}
       =f_8+P^6WU_2b.
\end{equation}
Both terms on the right belong to $L$, since $W\in H$ and $U_2\in H$.
For an entirely integral check, use $J_3=J_2J_1$ and
\eqref{eq:secondrealframe}. The identity becomes
\begin{equation}\label{eq:hexadintegercheck}
 \bigl[-P^6J_2-iP^6J_3+P^5(\id+J_1)-P^3(J_1+J_2)
       -2P^6WU_2\bigr]B=8\mathbf e_{1,\,6}.
\end{equation}
Substitution of the displayed signed row maps and the seed $B$ gives
zero in every row except row $1$, where the entry in real position $6$
is $8$. This proves \eqref{eq:hexadembedding}.

Changing a sign in the numerator changes its half by an element of
$F_\lambda$. Thus \eqref{eq:hexadembedding} realizes the glue word
$\mathbf q$. By Lemma~\ref{lem:secondframe}, the
isometries $\mathcal J$, $\GT$ preserve this line of the second frame
and act on its indices as on the code. They realize its whole permutation orbit
up to frame vectors. Lemma~\ref{lem:hexadorbit} realizes every word
of $\cC_\lambda$, proving $\Lambda_\lambda\subseteq L$.
The reflection assertion is Lemma~\ref{lem:modularreflection}.
\end{proof}

Use this same line for the generating reflection. The corresponding
complex subspaces of the two frames are
\begin{equation}\label{eq:twotwelvespaces}
\begin{aligned}
 V_\lambda&=\langle P^rJ_sf:r\in\mathcal Q,\ 0\le s\le3\rangle_{\C},\\
 V'_\lambda&=\langle P^{-r}J_sb:r\in\mathcal Q,\ 0\le s\le3\rangle_{\C}.
\end{aligned}
\end{equation}
The coordinate row sign and the new reflection are, respectively,
\begin{equation}\label{eq:tworeflections}
 D=E_2=2\pr_{V_\lambda}-\id\in H,\qquad
 r_*=2\pr_{V'_\lambda}-\id\in\Aut_{\Z[i]}(L,\,h).
\end{equation}
Indeed $D$ has row signs $(-1)^{2\cdot p}$, and
Proposition~\ref{prop:secondframereflection} proves the assertion for
$r_*$. Both reflections have positive and negative eigenspaces of complex
dimensions $12$ and $16$, and are determined by the two subspaces.

Put
\[
 G=\langle H,\,r_*\rangle\le\Aut_{\Z[i]}(L,\,h).
\]
This group is finite: the isometry group of a positive-definite lattice
acts faithfully on a finite norm-bounded set containing a lattice basis.
We next recover the frame exchange from $D$ and $r_*$.

\subsection{Recovering the frame exchange}\label{sec:preservation}
Because the two labeled frames are orthogonal with the same norms, there
is a unique ambient unitary map satisfying
\begin{equation}\label{eq:S}
 S(P^rJ_sf)=P^{-r}J_sb
 \qquad(0\le r\le6,\ 0\le s\le3).
\end{equation}
In particular $S(F)=F'$ and $S(V_\lambda)=V'_\lambda$. Conjugating orthogonal
projection by this unitary map gives
\begin{equation}\label{eq:ambientreflection}
 r_*=SDS^{-1}.
\end{equation}
At this point the map $S$ need not be a lattice automorphism. The
integral operators are $D$ and $r_*$. To express $S$ in terms of them
we verify two facts about the overlap of the frames: $S$ is a
self-adjoint involution of trace four, and its compression to $V_\lambda$
satisfies one quadratic equation. The remainder of the recovery is a
two-involution argument.

\begin{lemma}[The frame-overlap operator]\label{lem:unitary}
The map $S$ is a self-adjoint unitary involution centralizing $\mathcal J$, and
\[
 \operatorname{tr}_{\C}S=4,\qquad
 SP=P^{-1}S,\qquad S\GT=\GT S.
\]
\end{lemma}
\begin{proof}
Unitarity and the commutations follow from the compatible labeled
frames in Lemma~\ref{lem:secondframe}; the shift relation follows from
\eqref{eq:S}. For the remaining assertions use the Hermitian blocks $\mathsf Q_k$
in \eqref{eq:Pauliblocks}.
The correspondence \eqref{eq:S} now reads as the reverse-circulant formula
\begin{equation}\label{eq:Sblocks}
 (Sv)_r=\frac14\sum_{s=0}^6 \mathsf Q_{r+s}v_s.
\end{equation}
The blocks are Hermitian and their pattern is symmetric in $r$, $s$,
so $S^*=S$ and unitarity gives $S^2=\id$. Finally,
$\operatorname{tr}_{\C}S=\frac14\sum_k\operatorname{tr}\mathsf Q_k
=\frac14(4+3\cdot4)=4$.
\end{proof}
The same block pattern gives the following overlap and the local transport
identities in Section~\ref{sec:geometryorder}.

\begin{lemma}[The three-block overlap]\label{lem:compression}
Let $A_D=\pr_{V_\lambda}S|_{V_\lambda}$ be the compression of the frame
comparison to the space fixed by $D$. Then
\begin{equation}\label{eq:compressionquadratic}
 4A_D^2-2A_D-\id=0.
\end{equation}
\end{lemma}
\begin{proof}
In the row order $1$, $2$, $4$, the corresponding corner of
\eqref{eq:Sblocks} is
\begin{equation}\label{eq:reflectioncompression}
 4A_D=
 \begin{pmatrix}
 \id-K_2&K_1-K_4&K_4-K_2\\
 K_1-K_4&\id-K_4&K_2-K_1\\
 K_4-K_2&K_2-K_1&\id-K_1
 \end{pmatrix}.
\end{equation}
The symmetry $\GT$ cycles these three rows and conjugates $K_a$ to
$K_{2a}$. Together with self-adjointness, this reduces
$(4A_D)^2=2(4A_D)+4$ to one diagonal and one off-diagonal block.
The identities $K_a^2=\id$ and pairwise anticommutation give
\begin{align*}
 &(\id-K_2)^2+(K_1-K_4)^2+(K_4-K_2)^2=6-2K_2,\\
 &(\id-K_2)(K_1-K_4)+(K_1-K_4)(\id-K_4)\\
 &\hspace{28mm}+(K_4-K_2)(K_2-K_1)=2(K_1-K_4).
\end{align*}
These are the corresponding blocks of $2(4A_D)+4$.
Dividing by four proves \eqref{eq:compressionquadratic}.
\end{proof}

\begin{lemma}[Two-involution recovery]\label{lem:dihedralrecovery}
One has
\begin{equation}\label{eq:reflectiongeneration}
 (DS)^5=-\id,\qquad S=-D(Dr_*)^3.
\end{equation}
Consequently $S$ preserves $L$ and
$G=\langle H,\,r_*\rangle=\langle H,\,S\rangle$.
\end{lemma}
\begin{proof}
Put $p(t)=4t^2-2t-1$.
The compression identity gives $p(A_D)=0$.
If $v\in V_\lambda$ and $Sv\in V_\lambda$, then $A_D^2v=v$; since
$p(1)=1$ and $p(-1)=5$, this forces $v=0$.
Thus $V_\lambda\oplus SV_\lambda$ has dimension $24$. It is invariant under $S$ and
under $D$, since $DSv=2A_Dv-Sv$ for $v\in V_\lambda$.
The operator $S$ exchanges its two summands, so has trace zero there.
On its four-dimensional orthogonal complement, $D=-\id$; the total
trace $\operatorname{tr}S=4$ forces the self-adjoint involution $S$
to be $\id$ on this complement.

Put $u=DS$. On $V_\lambda$ one has
$(u+u^{-1})v=2A_Dv$. The operator $u+u^{-1}=DS+SD$ commutes with $S$,
so $p((u+u^{-1})/2)=0$ on $V_\lambda\oplus SV_\lambda$.
The Laurent polynomial identity
\[
\begin{aligned}
 t^2p\!\left(\frac{t+t^{-1}}2\right)
   &=t^4-t^3+t^2-t+1,\\
 (t+1)(t^4-t^3+t^2-t+1)&=t^5+1
\end{aligned}
\]
gives $u^5=-\id$ there. The same holds on the orthogonal complement,
where $u=-\id$. Hence $(DS)^5=-\id$ on the whole space.

Now $r_*=SDS$, so $Dr_*=(DS)^2$ and
\[
 -D(Dr_*)^3=-D(DS)^6=S.
\]
Every factor on the left preserves $L$, so $S\in\langle H,\,r_*\rangle$.
Conversely $r_*=SDS\in\langle H,\,S\rangle$, giving equality of the groups.
\end{proof}

\begin{corollary}\label{cor:reflectionoutsideH}
The reflection $r_*$ does not belong to $H$.
\end{corollary}
\begin{proof}
If $r_*\in H$, the recovery formula would give $S\in H$, since
$-\id\in H$ and $D\in H$. But $S(F)=F'$ is disjoint from $F$, whereas
$H$ stabilizes $F$.
This is a contradiction.
\end{proof}

\begin{proof}[Completion of Theorem~\ref{thm:main}]
The code, lattice, and full coordinate stabilizer were established
before the second frame was constructed. The self-dual three-block
glue and Lemma~\ref{lem:modularreflection} give the additional
integral reflection. Lemmas~\ref{lem:unitary} and~\ref{lem:dihedralrecovery}
then prove the assertions about $S$, including lattice preservation;
Lemma~\ref{lem:secondframe} gives $S(F)\cap F=\varnothing$.
\end{proof}

In the counting argument $S$ denotes the derived word $-D(Dr_*)^3$.
Corollary~\ref{cor:intrinsicgenerator} later identifies both $r_*$ and
$-S$ as reflections in intrinsic dozens.

\section{Counting canonical frames and the full lattice group}\label{sec:geometryorder}
The code gives $4060$ crosses and the stabilizer $H$ of $F$.
We prove transitivity by joining two coordinate orbits, determine the
quartets through a couple, and count canonical frames by incidence.
Their intrinsic definition and stabilizer then determine the lattice group.

\subsection{The intrinsic configurations}\label{sec:configurations}
A \emph{couple} is an unordered pair of distinct orthogonal crosses.
Recall that a frame consists of $28$ mutually orthogonal crosses.
We use the following form of Conway's quartet--dozen--base geometry
\cite[pp.~76--79]{Conway}.

\begin{definition}[Quartets]\label{def:quartet}
A \emph{quartet} is a set $Q$ of four mutually orthogonal minimal crosses
such that every cross outside $Q$ which is orthogonal to three members
is orthogonal to the fourth. Equivalently, the Hermitian support of an
external minimal vector on $Q$ never has size one.
\end{definition}
The exclusion of the four members of $Q$ is essential. A coordinate row
is a quartet: every other coordinate cross is orthogonal to the row,
and every noncoordinate minimum has zero, two, or four nonzero Gaussian
coordinates in it. In particular $Q_1$ is a quartet.

For a couple $\kappa$ and a quartet $Q$, set
\begin{align}
 \Delta(\kappa)&=\bigcup_{\substack{Q'\text{ a quartet}\\\kappa\subset Q'}}Q',
       \label{eq:intrinsicdozen}\\
 F(Q)&=\bigcup_{\kappa\in\binom Q2}\Delta(\kappa).
       \label{eq:canonicalbase}
\end{align}
We shall prove that $\Delta(\kappa)$ is a mutually orthogonal twelve-set,
called its \emph{dozen}, and that $F(Q)$ is a frame. These
are the \emph{canonical frames} (Conway's bases); write
$\cF=\{F(Q):Q\text{ a quartet}\}$.

Every Gaussian lattice isometry preserves the minimal shell and its
orthogonality relation. It therefore preserves quartets, and
\begin{equation}\label{eq:configurationequivariance}
 a\Delta(\kappa)=\Delta(a\kappa),\qquad
 aF(Q)=F(aQ)
 \quad\bigl(a\in\Aut_{\Z[i]}(L,\,h)\bigr).
\end{equation}

\subsection{The monomial point stabilizer}\label{sec:codeinput}
Let $p_0=6$ be the label of row $0$. The coordinate cross $x=X_1$
corresponds to the pair $\{0,\,p_0\}$ in its inner binary labels.
An element of $P_0$ fixes this cross exactly when
\[
 \tau p_0\in\langle p_0\rangle.
\]
This imposes two linear conditions on $\tau$, so $|P_{0,\,x}|=2^{11}$.
The map $H_x\longrightarrow\operatorname{GL}_3(2)_{p_0}$ is onto: any lift
of a row stabilizer can be adjusted by a translation to fix the chosen cross.
The quotient acts transitively on the three other coordinate crosses in that
row, on the other six rows, and on the three nonzero forms $\ell$ satisfying
$\ell(p_0)=0$.

For each fixed four-block label, $P_0$ is transitive on
$8\cdot32=256$ minimal vectors, hence on $64$ crosses. For each fixed
heavy label, it is transitive on $128\cdot16=2048$ vectors, hence on
$512$ crosses. These are the fixed-label transitivity statements of
Lemmas~\ref{lem:foursupports} and~\ref{lem:heavysupports}.

\begin{proposition}\label{prop:localorbits}
Put $Y=x^\perp\cap\cX$. Its four $H_x$-orbits are
\[
 Y=O_3\sqcup O_{24}\sqcup O_{192}\sqcup O_{1536},
\]
where $O_3=Q_1\setminus\{x\}$, $O_{24}=F\setminus Q_1$,
and $O_{192}$, $O_{1536}$ are the four-block and heavy crosses orthogonal
to $x$, respectively. Subscripts give the orbit sizes.
\end{proposition}
\begin{proof}
The two coordinate orbits have sizes $3$, $24$: the quotient acts on the three
nonzero cosets of $\langle p_0\rangle$ in the fixed row, while for $q\ne p_0$
the values $\tau q$ remain arbitrary under the condition on $\tau p_0$.

In the four-block family, a row containing $x$ is active exactly when
$\ell(p_0)=1$. Each active support contains one real position in every
Gaussian pair, so orthogonality is equivalent to $\ell(p_0)=0$.
There are three such labels, eight supports for each, and $32$ signed lifts
on each support. This gives $3\cdot8\cdot32/4=192$ crosses.
The subgroup $P_{0,\,x}$ alone changes the four active constants by linear forms $m$ with $m(p_0)=0$.
Its support orbit has size four, not eight.

To join the two support orbits, use the monomial maps
\begin{equation}\label{eq:couplemonomialgens}
 \rho=P^{-2}R_0P,\qquad \beta=\rho^{-1}R_0\rho^2.
\end{equation}
Both $R_0$ and $\beta$ fix $x$, $y$; their row permutations are
$(1\ 2)(3\ 6)$ and $(1\ 6\ 3\ 2)(4\ 5)$.
For later use, record their action on all six support states.
In row $p$ the selected plane is $\{e\in E:\ell(e)=1+\epsilon_p\}$.
For each $\ell\in\{1,\,6,\,7\}$ let $p$ be the smallest active row label.
The parity $\epsilon_p+\epsilon_{p+p_0}$ distinguishes the two
$P_{0,\,x}$-orbits and is unchanged by Gaussian scalars.
Number the pairs $(\ell,\,\epsilon_p+\epsilon_{p+p_0})$ as
\[
 1=(1,\,0),\ 2=(1,\,1),\ 3=(6,\,0),\ 4=(6,\,1),\,\
 5=(7,\,0),\ 6=(7,\,1).
\]
Substitution of the signed row maps gives
\begin{equation}\label{eq:localsixstates}
 R_0:(3\ 5)(4\ 6),\qquad
 \beta:(1\ 2)(3\ 6\ 4\ 5).
\end{equation}
In particular $\beta$ interchanges the two states with label $\ell=1$.
Together with the quotient action on the three labels, this proves
support transitivity in $H_x$.

For sign transitivity, the maps
$k_1=U_4U_5$ and $k_2=I U_6J_1 U_1$ fix $x$ and the support
of $v=P^3A$. Their sign differences from $v$ have common moments
$6$ and $2$. These are independent in $1^\perp$, and the moment
kernel is the even block-constant sign code, of order eight.
As in Lemma~\ref{lem:foursupports}, the moment action is additive.
Hence all $32$ lifts occur in the point stabilizer. Combined with
support transitivity, this proves the $192$-orbit.

For the heavy family, $c=p_0$ cannot be orthogonal to $x$, because the heavy
block meets every Gaussian pair. Let $c\ne p_0$. In row $p_0$ the supports run equally through the four affine lines parallel
to $c$. Two of these lines avoid both positions $0$, $p_0$, since $c\ne p_0$.
Thus half the $128$ supports avoid $x$, giving $64$ supports.
The stabilizer of a heavy support in the nine-dimensional translation space is
\[
 \tau=c\ell^{\top},\qquad \ell(c)=0.
\]
Imposing $\tau p_0\in\langle p_0\rangle$ adds $\ell(p_0)=0$; this leaves dimension
one since $c$, $p_0$ are independent. The point-fixing translation space has
dimension seven, so its orbit has $2^6=64$ supports. The sixteen signs from
$E_C$ are transitive on each fibre. Thus there are $64\cdot16/4=256$ crosses
per heavy label, and the quotient fuses the six labels. This gives $1536$.
\end{proof}

In particular,
\begin{equation}\label{eq:1755}
 |Y|=3+24+192+1536=1755,\qquad
 |H_x|=\frac{|H|}{28}=2^{14}\cdot3,\qquad
 |H_{x,\,y}|=\frac{|H|}{28\cdot3}=2^{14}.
\end{equation}

\subsection{Coordinate quartets and a local move}\label{sec:five}
In row $p$, a coordinate cross corresponds to a coset of $\langle p\rangle$
in $E$. Thus two crosses in that row correspond together to an affine plane.

\begin{lemma}[Coordinate quartets]\label{lem:coordinatequartets}
A four-set in $F$ is a quartet precisely when it is a coordinate row, or
consists of two crosses in each of two distinct rows $p$, $p'$, with each pair
corresponding to a coset of $\langle p,\,p'\rangle$ in $E$.
There are $91$ coordinate quartets, with $13$ through each coordinate cross.
\end{lemma}
\begin{proof}
A four-block minimum is nonzero in every Gaussian coordinate of its four
active rows. A heavy minimum of label $c$ is nonzero throughout row $c$;
in every other row $p$ its support is an affine $c$-line, meeting two
Gaussian coordinates. By Lemma~\ref{lem:heavysupports} every translation
$e\mapsto e+\tau p$ is available. At linearly independent row labels the
values $\tau p$ are arbitrary, so these $c$-lines can be prescribed independently.
We classify a four-set by its row multiplicities.

The pattern $(4)$ is a row quartet. For $(3,\,1)$, a four-block label
vanishing on the first row and nonzero on the second gives support one.

For $(2,\,2)$, let the pair differences in rows $p$, $p'$ be
$\delta\bmod\langle p\rangle$ and $\delta'\bmod\langle p'\rangle$.
For a heavy label $c\notin\{p,\,p'\}$, the contribution from row $p$ is always
one unless $c\in\{\delta,\,\delta+p\}$, when it is zero or two.
Independence of the two row translations therefore forces
\[
 \{\delta,\,\delta+p\}\setminus\{p'\}
 =\{\delta',\,\delta'+p'\}\setminus\{p\}.
\]
The two pairs have different sums $p$, $p'$, so this equality forces
$\{\delta,\,\delta+p\}=\{p',\,p+p'\}$ and
$\{\delta',\,\delta'+p'\}=\{p,\,p+p'\}$, the asserted plane condition.
Conversely, under that condition a heavy label outside $\{p,\,p'\}$
gives contributions $(1,\,1)$ unless $c=p+p'$, when both are even.
For $c=p$ or $p'$ one contribution is two and the other is even.
Four-block contributions are also even. Thus these sets are quartets.

For $(2,\,1,\,1)$ on rows $p$, $q$, $r$, choose a heavy label outside
$\{p,\,q,\,r,\,\delta,\,\delta+p\}$. Its contribution on the double row is always one;
the two single rows are independent, so their $c$-lines can both avoid
the chosen crosses. This gives support one.

For $(1,\,1,\,1,\,1)$, a row label outside the span of the other three is
separated from them by a four-block label. Otherwise, after numbering,
$p_1$, $p_2$, $p_3$ are independent and $p_4=p_1+p_2+p_3$.
Choose $c$ outside these four labels and pass to $E/\langle c\rangle$.
Starting with one heavy support, translate its lines by $\tau p_i$.
Hitting the first cross and avoiding the next two restrict each
$\overline{\tau p_i}$, $i=1$, $2$, $3$, to an affine line of direction
$\overline p_i$. These three directions span $E/\langle c\rangle$, so their
sum ranges over all of $E/\langle c\rangle$. Since $\tau p_4=\tau p_1+\tau p_2+\tau p_3$,
one can also avoid the fourth cross, again giving support one.
This exhausts the patterns. Finally there are
$7+\binom72\cdot2^2=91$ quartets, and transitivity of $H$ on $F$ gives
$4\cdot91/28=13$ through each cross.
\end{proof}

Let $\DD$ be the twelve coordinate crosses in rows $0$, $4$, $5$, of binary
labels $6$, $4$, $2$. The lemma gives exactly five coordinate quartets through
$x$, $y$:
\begin{equation}\label{eq:fivequartets}
\begin{aligned}
 Q_1&=\{X_1,\,X_2,\,X_3,\,X_4\},\\
 Q_2&=\{X_1,\,X_2,\,X_{17},\,X_{19}\},&
 Q_3&=\{X_1,\,X_2,\,X_{18},\,X_{20}\},\\
 Q_4&=\{X_1,\,X_2,\,X_{21},\,X_{24}\},&
 Q_5&=\{X_1,\,X_2,\,X_{22},\,X_{23}\}.
\end{aligned}
\end{equation}
Their union is $\DD$, the \emph{reference dozen}; its intrinsic
identification will follow by excluding noncoordinate completions.
Set
\begin{equation}\label{eq:ktdef}
 g=SUS,\qquad t=gVgW.
\end{equation}
The seed identity $UB=-B$ and $Sf=b$ give $gf=-f$, so $g$ is an
involution in $G_x$. Using \eqref{eq:Sblocks} and Table~\ref{tab:UVW}
gives the three local images
\begin{equation}\label{eq:onebridge}
 tf_1=-f_1,\qquad tf_2=-f_2,\qquad tf_3=-if_{20}.
\end{equation}
Thus $t\in G_{x,\,y}$ joins $O_3$ to $O_{24}$ and does not belong to $H$:
an element of $H$ fixing $x$ preserves its coordinate row.

\subsection{Arithmetic fusion and transitivity}\label{sec:coupletrans}
The rank-$56$ integral representation rules out partial orbit fusions.
Every subgroup considered here is finite and acts on $L\otimes_{\Z}\mathbb Q$.

\begin{lemma}[Cyclotomic prime obstruction]\label{lem:cyclotomicprime}
If a prime $p$ divides the order of a finite group of automorphisms
of a rank-$n$ lattice, then $p-1\le n$.
\end{lemma}
\begin{proof}
Cauchy's theorem gives an element of order $p$. It has a nontrivial
$p$th root of unity as an eigenvalue, whose minimal polynomial over
$\mathbb Q$ is $\Phi_p$ of degree $p-1$. This polynomial divides its
rational characteristic polynomial of degree $n$.
\end{proof}

\begin{proposition}[One fusion suffices]\label{prop:arithmeticfusion}
Let
\[
 H_x\le K\le\Stab_{\Aut_{\Z[i]}(L,\,h)}(x).
\]
If a $K$-orbit meets both $O_3$ and $O_{24}$, then $K$ is transitive
on $Y=x^\perp\cap\cX$.
\end{proposition}
\begin{proof}
Because $H_x\le K$, each $K$-orbit is a union of the four $H_x$-orbits
in Proposition~\ref{prop:localorbits}. An orbit meeting both coordinate
orbits contains their union. Its only possible sizes are therefore
\[
 \begin{array}{c|c}
 \text{orbits included beyond }O_3\cup O_{24}&\text{size}\\\hline
 \text{none}&27\\
 O_{192}&27+192=219=3\cdot73\\
 O_{1536}&27+1536=1563=3\cdot521\\
 O_{192},\,O_{1536}&1755.
 \end{array}
\]
An orbit length divides $|K|$. The middle two possibilities would thus
force elements of prime orders $73$ and $521$, respectively, contrary
to Lemma~\ref{lem:cyclotomicprime}, since $72>56$ and $520>56$.

If the orbit had size $27$, it would be exactly
$O_3\cup O_{24}=F\setminus\{x\}$. Since $K$ fixes $x$, it would then
preserve $F$. Corollary~\ref{cor:framestabilizer}, which determined the
full frame stabilizer from the code, gives
\[
 K\le\Stab_{\Aut_{\Z[i]}(L,\,h)}(F)\cap
        \Stab_{\Aut_{\Z[i]}(L,\,h)}(x)=H_x.
\]
This contradicts fusion of two distinct $H_x$-orbits. Only size $1755$
remains, and that is all of $Y$.
\end{proof}

\begin{proposition}\label{prop:coupletrans}
The group $G$ is transitive on crosses and on ordered orthogonal pairs
of crosses. More precisely, $\langle H_x,\,t\rangle$ is transitive on $Y$,
and $\langle H,\,t\rangle$ already has both global transitivity properties.
\end{proposition}
\begin{proof}
The identities \eqref{eq:onebridge} supply the hypothesis of
Proposition~\ref{prop:arithmeticfusion} for $K=\langle H_x,\,t\rangle$.
Thus $K$ is transitive on $Y$.

The shell $Y$ meets each of the three $H$-orbits of crosses: its
coordinate part is nonempty, as are its four-block and heavy parts.
All these crosses lie in one $\langle H,\,t\rangle$-orbit. Since $H$ is
transitive on each family by Theorem~\ref{thm:shellproved}, that orbit
contains all $4060$ crosses, proving cross transitivity.

Given two ordered orthogonal pairs, first take the first member of
each pair to $x$ by cross transitivity. Their second members are then
in $Y$, and an element of $K$ takes one to the other while fixing $x$.
Composing these maps proves transitivity on ordered pairs. In
particular it proves transitivity on unordered couples.
\end{proof}

\subsection{The common perpendicular shell}\label{sec:fivebases}
We continue with the standard couple $x=X_1$, $y=X_2$ and the reference
dozen $\DD$. The couple's four real positions form the plane
\[
 E_{x,\,y}=\langle2,\,4\rangle=\{0,\,2,\,4,\,6\}\subset E.
\]
Put $V_1=\langle\DD\rangle_{\C}$. It has complex dimension $12$,
with complement $V_1^\perp$ of dimension $16$.
In a span or an orthogonality condition, a cross denotes its complex line.

\begin{lemma}\label{lem:localshell}
The crosses orthogonal to both $x$ and $y$ split into
\begin{equation}\label{eq:730}
 10+80+640=730.
\end{equation}
The first ten are $\DD\setminus\{x,\,y\}$. The next eighty form the set
$Y_0$ of minimal crosses in $V_1^\perp$; the remaining $640$ are
\emph{mixed}, with nonzero projections to both $V_1$ and $V_1^\perp$.
\end{lemma}
\begin{proof}
There are $26$ other coordinate crosses. A four-block minimum avoids $x$, $y$
exactly when its label $\ell$ satisfies $\ell(p_0)=0$: an active row meets
both Gaussian coordinates. There are three such nonzero labels, giving
$3\cdot64=192$ crosses. The normal of $E_{x,\,y}$ gives $64$ crosses in $V_1^\perp$;
the other two labels give $128$ mixed crosses, each occupying two rows of
$V_1$ and two of $V_1^\perp$.

A heavy minimum cannot have heavy label $p_0=6$. In row $p_0$ its support
is an affine line parallel to the heavy label $c$. Such a line avoids $E_{x,\,y}$
precisely when $c=2$ or $4$ and the line is one of the two lying outside
$E_{x,\,y}$. The translation action distributes the $128$ supports of each heavy
label equally over its four affine lines in row $p_0$. Thus each of these
two labels gives $64$ real supports, each with sixteen signed lifts, or
$256$ crosses. Each meets both $V_1$ and $V_1^\perp$.
Consequently
\[
 26+192+512=730,\qquad 80=16+64,\qquad 640=128+512,
\]
as required.
\end{proof}

\subsection{The local orbits of the couple stabilizer}\label{sec:coupleorbits}
Recall the direction plane $E_{x,\,y}=\langle2,\,4\rangle$ of the
coordinate couple $\{x,\,y\}$.
The row image of $H_{x,\,y}$ is the order-eight subgroup fixing $p_0$ and $E_{x,\,y}$,
generated by $R_0$, $\beta$. Its row kernel is $P_{0,\,x}$: once a
translation fixes one Gaussian coordinate in row $p_0$, it fixes all four.

\begin{lemma}\label{lem:B0orbits}
The $H_{x,\,y}$-orbits on $Y$ have sizes
\[
 1,\,\quad2,\,\quad8,\,\quad16,\,\quad64,\,\quad128,\,\quad512,\,\quad1024.
\]
The first four are coordinate orbits, the next two four-block orbits,
and the last two heavy orbits.
\end{lemma}
\begin{proof}
The cross $y$ is fixed, and the other two coordinate crosses in row $0$
form an orbit. The other rows in $E_{x,\,y}$ form a two-element orbit and those
outside $E_{x,\,y}$ a four-element orbit. Translations give all four coordinate
crosses in each such row.

For four-block crosses orthogonal to $x$, the possible labels are $\ell=1$, $6$, $7$.
The group $P_{0,\,x}$ has two orbits of $32$ for each label, as proved in
Proposition~\ref{prop:localorbits}. The action \eqref{eq:localsixstates} on these six states gives
orbits $2\cdot32=64$ and $4\cdot32=128$.

For each heavy label $c\ne p_0$, the point-fixing row kernel is transitive
on the $256$ orthogonal crosses of that label, again by
Proposition~\ref{prop:localorbits}. The row group has two orbits on these
labels, $E_{x,\,y}\setminus\{0,\,p_0\}$ and $E\setminus E_{x,\,y}$, of sizes two and
four. This gives $512$, $1024$. These exhaust $Y$.
\end{proof}

Two further row-energy calculations from \eqref{eq:Sblocks} give
\begin{equation}\label{eq:tlocalcolumns}
\begin{aligned}
 \bigl(4\|(tf_5)_r\|^2\bigr)_{r=0}^6&=(0,\,4,\,4,\,4,\,0,\,0,\,4),\\
 \bigl(4\|(t(P^5A/2))_r\|^2\bigr)_{r=0}^6&=(2,\,2,\,2,\,2,\,2,\,4,\,2).
\end{aligned}
\end{equation}
In the group $G_0=\langle H_{x,\,y},\,t\rangle$, these join the
$H_{x,\,y}$-orbits of sizes $16$, $64$ and $128$, $512$, respectively;
\eqref{eq:onebridge} joins those of sizes $2$, $8$.
The $1024$-orbit is exactly the part of $Y\setminus\{y\}$ nonorthogonal
to $y$: a heavy $c$-line in row $p_0$ avoiding $x$ meets $y$ precisely
when $c\notin E_{x,\,y}$. The other seven orbits are therefore $\{y\}$
and the common perpendicular shell.

\subsection{The five quartets through a couple}\label{sec:quartets}
\begin{lemma}[A triple in a frame]\label{lem:tripleframe}
If three members of a quartet lie in a frame $\Phi$, then its fourth
member also lies in $\Phi$.
\end{lemma}
\begin{proof}
Suppose the fourth member $w$ of the quartet $Q$ is outside $\Phi$.
Every member of $\Phi\setminus(Q\setminus\{w\})$ is external to $Q$
and orthogonal to its other three members, hence to $w$ as well.
Thus $w$ is orthogonal to all of $\Phi$, which spans the space,
a contradiction.
\end{proof}

\begin{lemma}\label{lem:quartetexhaustion}
The quartets containing $x$, $y$ are precisely $Q_1$, $\ldots$, $Q_5$ in
\eqref{eq:fivequartets}. Every couple lies in exactly five quartets,
and $\Delta(\{x,\,y\})=\DD$.
\end{lemma}
\begin{proof}
If a third member is coordinate, Lemma~\ref{lem:tripleframe} makes the
fourth coordinate as well. Lemma~\ref{lem:coordinatequartets} then gives
exactly the five displayed quartets.

The noncoordinate candidates form the $H_{x,\,y}$-orbits of sizes
$64$, $128$, $512$. By \eqref{eq:tlocalcolumns}, every member of the first
can be taken to $X_5$, and every member of the other two to $X(P^5A/2)$,
by an element of $G_0$ fixing $x$, $y$. The triple
$\{X_1,\,X_2,\,X_5\}$ lies in no coordinate quartet by
Lemma~\ref{lem:coordinatequartets}, and hence in no quartet by
Lemma~\ref{lem:tripleframe}. For the other representative the block
formulas give
\begin{equation}\label{eq:basictripletransport}
 tf_{19}=-if_3,\qquad tf_6=-iP^3A/2.
\end{equation}
Together with $tX_1=X_1$, these take $\{X_1,\,X_6,\,X_{19}\}$ to
$\{X_1,\,X_3,\,X(P^3A/2)\}$; applying $\GT^2$ gives
$\{X_1,\,X_2,\,X(P^5A/2)\}$. The original triple lies in three different
coordinate rows, so the same two lemmas exclude a quartet through it. Thus no noncoordinate candidate occurs.
Couple transitivity transports exhaustion to every couple, and taking
the union gives $\Delta(\{x,\,y\})=\DD$.
\end{proof}

\begin{proposition}\label{prop:couplelocalorbits}
The group $G_0=\langle H_{x,\,y},\,t\rangle$ has orbits on $Y$
of sizes $1$, $10$, $80$, $640$, $1024$. It is transitive on each of the three sets
in \eqref{eq:730}.
\end{proposition}
\begin{proof}
The preceding lemma makes $\DD$ intrinsic, so $G_0$ preserves it,
its span, and the perpendicular shell $Y_0$. The fusions already proved
join $2+8$, $16+64$, and $128+512$; their unions are exactly
$\DD\setminus\{x,\,y\}$, $Y_0$, and the mixed set. The $1024$-orbit is
distinguished by nonorthogonality to $y$. Thus no further fusion occurs.
\end{proof}

\subsection{From a quartet to its canonical frame}\label{sec:canonicalbases}
Every dozen consists of twelve mutually orthogonal crosses: this holds
for $\DD$ and follows in general by couple transitivity and equivariance.
Put $\Omega=\{Q_1,\,\ldots,\,Q_5\}$.

\begin{proposition}[Canonical frames]\label{prop:canonicalbases}
The group $G$ is transitive on quartets. For a quartet $Q$, its six
couples determine three distinct dozens whose union $F(Q)$
is a frame. These canonical frames form one $G$-orbit, with
$F(Q_1)=F$, and exactly seven quartets determine each frame.
Moreover $G_{Q_1}=H_{Q_1}$ and $G_0$ is transitive on $\Omega$.
\end{proposition}
\begin{proof}
First examine $Q_1$. A
row-preserving translation interchanges a couple with its complementary
couple and preserves the corresponding row dozen. The map $\GT$ fixes
$x$ and cycles the three partitions of $Q_1$ into couples. Their dozens
therefore have row sets
\begin{equation}\label{eq:threedozenrows}
 \{0,\,4,\,5\},\qquad\{0,\,1,\,3\},\qquad\{0,\,2,\,6\}.
\end{equation}
They intersect pairwise in $Q_1$ and their union is $F$. Thus
$F(Q_1)=F$, and equivariance and the full frame-stabilizer
calculation give $G_{Q_1}=H_{Q_1}$.

The $H_{x,\,y}$-orbits on $\Omega$ are $\{Q_1\}$ and the other four
quartets: its row group exchanges labels $2$, $4$, while a point-fixing
translation with $\tau(4)\notin E_{x,\,y}$ interchanges the two cosets in
each of those rows. Since $t\notin H$, it cannot fix $Q_1$.
Consequently $G_0$ is transitive on $\Omega$. Taking a couple
of any quartet to $\{x,\,y\}$ now proves quartet transitivity.

For $Q=aQ_1$, equivariance gives
$F(Q)=aF(Q_1)=aF$. This transports the three dozens,
their intersections, and their orthogonal union, proving the frame
assertion and transitivity on canonical frames. If $F(Q)=F$,
then $aF=F$ and hence $a\in H$. The possible $Q=aQ_1$ are exactly
the seven row quartets; conversely each determines $F$. Transitivity
transports this fibre size to every canonical frame.
\end{proof}
The divisor seven counts quartets \emph{determining} a frame, not all
quartets \emph{contained} in it: $F$ contains $91$, but only its seven
row quartets determine $F$.

\begin{corollary}\label{cor:localgenerators}
The local stabilizers satisfy
\[
 G_{x,\,y}=G_0=\langle H_{x,\,y},\,t\rangle,\qquad
 [G_{x,\,y}:H_{x,\,y}]=5,\qquad
 G_x=\langle H_x,\,t\rangle=\langle H_x,\,g\rangle.
\]
\end{corollary}
\begin{proof}
The action of $G_{x,\,y}$ on $\Omega$ has stabilizer $H_{x,\,y}$ at $Q_1$,
since $G_{Q_1}=H_{Q_1}$. Its subgroup $G_0$ is transitive on $\Omega$
and contains this stabilizer, so equals $G_{x,\,y}$ and has the stated
index. By Proposition~\ref{prop:coupletrans}, $\langle H_x,\,t\rangle$
is transitive on $Y$ and contains its full point stabilizer $G_{x,\,y}$,
so equals $G_x$. Finally $g\in G_x$ and $t=gVgW$ with $V\in H_x$ and $W\in H_x$.
\end{proof}

\subsection{The elementary incidence count}\label{sec:incidencecount}
Let $N_2$ and $N_Q$ denote the numbers of couples and quartets.

\begin{corollary}[The canonical-frame count]\label{cor:canonicalcount}
There are $3\,562\,650$ couples, $2\,968\,875$ quartets, and
$424125$ canonical frames. Each cross lies in $2925$ canonical frames
and each couple in $45$.
\end{corollary}
\begin{proof}
\emph{Couples.}
There are $1755$ orthogonal crosses at the reference cross, by the
code calculation, and cross transitivity gives the same degree at
every cross. Hence there are $4060\cdot1755$ ordered orthogonal
pairs. Forgetting the ordering counts every couple twice, giving
\begin{equation}\label{eq:couplecount}
 N_2=\frac{4060\cdot1755}{2}=3\,562\,650.
\end{equation}

\emph{Quartets.}
Count incidences $(\kappa,\,Q)$ with $\kappa\subset Q$. The five-quartet theorem
gives five incidences for each couple, while each quartet contains
$\binom42=6$ couples. Thus $6N_Q=5N_2$.

\emph{Canonical frames.}
Every quartet determines one canonical frame, and
Proposition~\ref{prop:canonicalbases} gives exactly seven quartets
in each fibre. Thus $7|\cF|=N_Q$. Substitution yields
\begin{equation}\label{eq:totalbases}
\begin{aligned}
 N_Q&=\frac{5N_2}{6}
     =\frac{4060\cdot1755\cdot5}{2\cdot6}=2\,968\,875,\\
 |\cF|&=\frac{N_Q}{7}
     =\frac{4060\cdot1755\cdot5}{2\cdot6\cdot7}=424125.
\end{aligned}
\end{equation}

\emph{Incidences with a fixed cross or couple.}
Cross transitivity makes the number of frames through a cross
constant. Counting $(x,\,\Phi)$ with $x\in \Phi\in\cF$ gives
\[
 \#\{\Phi\in\cF:x\in \Phi\}=\frac{424125\cdot28}{4060}=2925.
\]
Similarly, couple transitivity makes the number through a couple
constant. Every frame contains $\binom{28}{2}$ couples, so
\[
 \#\{\Phi\in\cF:\kappa\subset \Phi\}
 =\frac{424125\binom{28}{2}}{3\,562\,650}=45.
\]

\end{proof}

\subsection{The full Gaussian isometry group}\label{sec:order}
\begin{theorem}\label{thm:order}
The group $G=\langle H,\,r_*\rangle=\langle H,\,S\rangle$ is the full
Gaussian isometry group, and
\begin{equation}\label{eq:orderfinal}
 G=\Aut_{\Z[i]}(L,\,h),\qquad
 |G|=424125\,(2^{13}\cdot168)=583\,704\,576\,000.
\end{equation}
Its quotient by the four Gaussian scalars has order
$145\,926\,144\,000$.
\end{theorem}
\begin{proof}
First we prove fullness. Let $a\in\Aut_{\Z[i]}(L,\,h)$ be arbitrary.
It preserves the quartet condition, so $aQ_1$ is a quartet. By
\eqref{eq:configurationequivariance},
\[
 aF=aF(Q_1)=F(aQ_1).
\]
Thus $aF$ lies in the intrinsic family $\cF$.
Proposition~\ref{prop:canonicalbases} proves that $G$ is transitive
on this family, so choose $a_0\in G$ with $a_0F=aF$.
Then $a_0^{-1}a$ fixes $F$. Its full stabilizer was determined
from the code in Corollary~\ref{cor:framestabilizer}, and equals
$H\subseteq G$. Hence $a_0^{-1}a\in G$, and therefore $a\in G$.
This proves $G=\Aut_{\Z[i]}(L,\,h)$.

We now use orbit--stabilizer on $\cF=G\cdot F$. Its size is
$424125$, by the incidence count, and its stabilizer is $H$, of the
independently determined order $2^{13}\cdot168$. Consequently
\[
 |G|=|\cF|\,|H|
     =424125\,(2^{13}\cdot168)=583\,704\,576\,000.
\]
The four Gaussian scalar units lie in $H$ and are central.
Division by four gives the quotient order.
\end{proof}

Orbit--stabilizer for crosses and ordered orthogonal pairs now gives
the stabilizer orders and factorizations used below:
\begin{equation}\label{eq:pointorders}
\begin{aligned}
 |G_x|&=\frac{|G|}{4060}=143\,769\,600,&
 |G_{x,\,y}|&=\frac{|G|}{4060\cdot1755}=81\,920,\\
 |G|&=2^{16}3^3 5^3\cdot7\cdot13\cdot29,&
 |(G/\mu_4)_x|&=2^{12}3^3 5^2\cdot13.
\end{aligned}
\end{equation}
In particular the projective cross stabilizer has order prime to $7$.

\subsection{Intrinsic reflections in dozens}\label{sec:intrinsicreflections}
Conway associates a reflection to a couple, fixing its dozen and
negating the complementary $16$-space \cite[p.~79]{Conway}; see also
\cite[\S1.2]{Wilson}. We now obtain the reflection in every couple and
identify both the constructed reflection $r_*$ and the frame exchange $S$
in this geometry.

\begin{theorem}\label{thm:couplereflections}
For every couple $\kappa$ the coordinate-free operator
\begin{equation}\label{eq:couplereflection}
 r_\kappa=2\pr_{\langle\Delta(\kappa)\rangle_{\C}}-\id
\end{equation}
is a Gaussian lattice isometry and an involution. Its eigenspaces have
complex dimensions $12$ and $16$, and $r_{a\kappa}=ar_\kappa a^{-1}$ for every
Gaussian lattice isometry $a$.
\end{theorem}
\begin{proof}
For $\kappa=\{x,\,y\}$ its dozen consists of rows with binary labels $6$, $4$, $2$.
The operator is therefore exactly $E_1\in E_C$, whose row signs are
$(-1)^{p_1}$. The sign-kernel theorem gives $E_1(L)=L$.
For a general couple, choose $a\in G$ taking the standard couple to it.
The intrinsic definition of the dozen gives $r_\kappa=aE_1a^{-1}$, proving
lattice preservation. The formula itself proves involutivity, the
eigenspace dimensions, independence of this choice, and equivariance
under the full Gaussian isometry group.
\end{proof}

\begin{corollary}[Intrinsic identification of both reflections]\label{cor:intrinsicgenerator}
Let $\kappa_\lambda=\{X_5,\,X_7\}=P\GT\{x,\,y\}$, and put
\[
 \kappa_*=S\kappa_\lambda
      =\{X(P^{-1}b),\,X(P^{-1}J_2b)\},\qquad
 \kappa^\circ=(DS)^2\kappa_\lambda.
\]
Then
\begin{equation}\label{eq:seeddozen}
 \Delta(\kappa_*)=
 \{X(P^{-r}J_sb):r\in\mathcal Q,\ 0\le s\le3\},
\end{equation}
and
\begin{equation}\label{eq:intrinsicgenerators}
 r_*=r_{\kappa_*},\qquad -S=r_{\kappa^\circ},\qquad
 G=\langle H,\,r_{\kappa_*}\rangle=\langle H,\,r_{\kappa^\circ}\rangle.
\end{equation}
Both reflections lie outside $H$. Moreover
\begin{equation}\label{eq:exchangedozen}
 \Delta(\kappa^\circ)=(DS)^2\Delta(\kappa_\lambda),\qquad
 \ker(S+\id)=\langle\Delta(\kappa^\circ)\rangle_{\C},\qquad
 F'=r_{\kappa^\circ}(F).
\end{equation}
\end{corollary}
\begin{proof}
The reference dozen $\DD$ has rows $0$, $4$, $5$. The map $P\GT$ takes
these to rows $1$, $2$, $4$, so equivariance gives
$\langle\Delta(\kappa_\lambda)\rangle_{\C}=V_\lambda$ and
$r_{\kappa_\lambda}=E_2=D$. The map $S$ takes these twelve crosses
to the set in \eqref{eq:seeddozen}, proving that identity.
Its span is $V'_\lambda$, so the defining projections give
$r_*=r_{\kappa_*}$.

The identities $(DS)^5=-\id$ and $D(DS)D=(DS)^{-1}$ give
\[
 (DS)^2D(DS)^{-2}=(DS)^4D=-(DS)^{-1}D=-S.
\]
Equivariance identifies the left side with $r_{\kappa^\circ}$ and
also gives the dozen identity in \eqref{eq:exchangedozen}.
The projection formula identifies its positive space with $\ker(S+\id)$.
Since $-1$ fixes every cross, $r_{\kappa^\circ}(F)=S(F)=F'$.
Nonmonomiality of $r_*$ is Corollary~\ref{cor:reflectionoutsideH};
that of $-S$ follows from $F'\ne F$. Finally $-\id\in H$ and
$G=\langle H,\,S\rangle$, giving the generating descriptions.

The couple $P\kappa_*$ is $\{X(b),\,X(J_2b)\}$. Since $\GT$ fixes $b$
and cycles $J_1$, $J_2$, $J_3$, conjugation by $P$ and then powers of $\GT$
shows that $H$ together with the reflection in any of these three seed
couples also generates $G$.
\end{proof}

\begin{remark}[The seed quartet and the second-frame code]\label{cor:secondframecode}
The second frame is also the canonical frame of the seed quartet:
\[
 F'=F(SQ_1)=F\bigl(\{X(J_sb):0\le s\le3\}\bigr).
\]
Indeed the quartet is $SQ_1$, and $F(Q_1)=F$, so this follows
from equivariance. With the chosen labels and phases,
$C(F')=C(F)=C$: the lattice isometry $S$ takes each real representative
$2\mathbf e_{r,\,j}$ to $b_{r,\,j}$. Thus the residue and torsion codes
are $\cD$, $\cC$; without labels and phases these statements are signed
coordinate equivalences. More generally, every canonical frame has
frame code equivalent to $C$, because the canonical frames form $G\cdot F$.

These canonical frames are $I$-compatible: their $56$ real coordinate directions are paired into
$28$ Gaussian crosses.
\end{remark}

\section{The Rudvalis group}\label{sec:simplicity}
Put
\[
 Z=\mu_4,\qquad R=G/Z,\qquad \overline H=H/Z,
 \qquad \overline{P}_0=P_0/Z.
\]
The simplicity argument uses the row-group structure obtained from the
code and the order theorem.

\subsection{The kernel on crosses}\label{sec:crosskernel}
\begin{lemma}[The kernel on an orthogonal shell]\label{lem:shellkernel}
An element of $G_x$ fixing every cross in $Y$ belongs to $Z=\mu_4$.
Equivalently, the kernel of $G_x\curvearrowright Y$ is $Z$.
\end{lemma}
\begin{proof}
Let $u\in G_x$ act trivially on $Y$. It fixes every coordinate cross:
it fixes $x$ by definition, and the other $27$ belong to $Y$.
Hence $u$ is Gaussian-diagonal.
Choose a heavy cross in $Y$, which exists by
Proposition~\ref{prop:localorbits}. Its heavy row is not row $0$;
its support therefore has two Gaussian positions in row $0$, both
different from $x$, and also meets the other six rows. Thus it meets
both coordinate $H_x$-orbits $O_3$ and $O_{24}$. The supports of its
$H_x$-translates consequently cover all $27$ coordinate positions
other than $x$.

Each such support has size $16$. Any two intersect, since
$16+16>27$. Fixing a cross forces all diagonal entries on its support
to be equal, and the intersections force the resulting values to
agree for all translates. The diagonal entries on all $27$ positions
therefore have a common value $\zeta\in\mu_4$.
Write $\zeta'\in\mu_4$ for the entry at $x$.  The seed $b=B/2$ has squared
norm four and squared coordinate norm $1/4$ at $x$, so integrality gives
\[
 h(b,\,ub)=\frac{15\zeta+\zeta'}{4}\in\Z[i].
\]
Thus $\zeta'-\zeta\in4\Z[i]$, which forces $\zeta'=\zeta$.
Thus $u$ is scalar. Conversely all four scalar units fix $x$ and every
cross in $Y$, so the kernel is $Z$.
\end{proof}

\begin{corollary}\label{cor:crosskernel}
The kernel of $G\curvearrowright\cX$ is $Z$. In particular $R$ acts
faithfully on the $4060$ crosses, and $G_x/Z$ acts faithfully on $Y$.
\end{corollary}
\begin{proof}
An element fixing every cross fixes $x$ and all of $Y$, so the lemma
puts it in $Z$. Conversely $Z$ fixes every cross. The shell assertion
is the lemma itself.
\end{proof}

\subsection{Primitivity from divisibility of suborbits}\label{sec:primitive}
Recall that $x=X_1$ lies in the row labelled $p_0=6$, and
$|P_{0,\,x}|=2^{11}$. The fixed-label $P_0$-orbits established in
Section~\ref{sec:codeinput} have sizes $64$ for four-block crosses and
$512$ for heavy crosses.

\begin{lemma}\label{lem:nonorthdivisibility}
Every $H_x$-orbit on the crosses nonorthogonal to $x$, other than $x$
itself, has size divisible by $64$.
\end{lemma}
\begin{proof}
There are no other coordinate crosses in this set. For a four-block cross
$v$, the stabilizer $P_{0,\,v}$ has order $2^{13}/64=2^7$.
Consequently
\[
 |P_{0,\,x}v|=\frac{2^{11}}{|P_{0,\,x}\cap P_{0,\,v}|}
\]
is divisible by $16$. Such a cross is nonorthogonal to $x$ exactly when
its label $\ell$ satisfies $\ell(p_0)=1$. Indeed its support in that row is
a plane of normal $\ell$ and meets exactly one position of the Gaussian pair
$\{0,\,p_0\}$. There are four such labels, permuted transitively by $H_x$:
the map onto $\operatorname{GL}_3(2)_{p_0}$ is surjective, and its dual action
is transitive on these four forms. Each $H_x$-orbit therefore meets all four
labels in equal numbers. Each intersection is a union of $P_{0,\,x}$-orbits,
so the whole orbit has size divisible by $4\cdot16=64$.

For a heavy cross $v$, the order of $P_{0,\,v}$ is $2^{13}/512=16$.
Thus every $P_{0,\,x}$-orbit on heavy crosses has size divisible by
$2^{11}/16=128$. An $H_x$-orbit is a union of these orbits, which proves the
remaining assertion.
\end{proof}

\begin{proposition}\label{prop:primitive}
The action of $R$ on $\cX$ is primitive.
\end{proposition}
\begin{proof}
The actions of $G$ and $R$ have the same blocks. Let $B$ be a block containing
$x$, and write $d=|B|$. Transitivity gives $d\mid4060$. A block containing
$x$ is invariant under $G_x$. By Proposition~\ref{prop:coupletrans}, either
it contains all $1755$ crosses orthogonal to $x$ or it contains none.
It is also invariant under $H_x$. Lemma~\ref{lem:nonorthdivisibility} gives
\begin{equation}\label{eq:blockdivisibility}
 d=1+1755\epsilon+64m,\qquad \epsilon\in\{0,\,1\},\quad m\ge0.
\end{equation}
If $\epsilon=0$, then $d$ is an odd divisor of $4060$ and $d\equiv1\pmod{64}$.
The odd divisors are
\[
 1,\,5,\,7,\,29,\,35,\,145,\,203,\,1015;
\]
only $1$ has that residue. Thus $B=\{x\}$.
If $\epsilon=1$, then $d\ge1756$. A proper divisor of $4060$ this large must
be $2030$, since its complementary divisor is an integer strictly between
$1$ and $4060/1756<3$. But \eqref{eq:blockdivisibility} makes $d$ divisible
by $4$, whereas $2030$ is not. Hence $d=4060$ and $B=\cX$.
\end{proof}

\subsection{Perfectness of the projective frame group}\label{sec:frameperfect}
Let
\[
 \overline E=(ZE_C)/Z\leq\overline{P}_0.
\]
The row-group calculation gives two exact sequences
\begin{equation}\label{eq:projectivefiltration}
\begin{aligned}
 1&\longrightarrow \overline E\longrightarrow\overline{P}_0
   \longrightarrow M_3(\F_2)/\langle \id\rangle\longrightarrow1,\\
 1&\longrightarrow\overline{P}_0\longrightarrow\overline H
   \longrightarrow\operatorname{GL}_3(2)\longrightarrow1.
\end{aligned}
\end{equation}
Here the additive group in the first line has order $2^8$, and
$\overline E\cong \cH_7/\langle\mathbf1\rangle$ has order $2^3$.
Indeed $I$ induces the translation $\tau=\id$, and
$Z\cap E_C=\{\pm \id\}$. On the two abelian factors the row quotient acts,
respectively, by permutation of the Hamming coordinates and conjugation of
matrices. The latter assertion follows directly from the flag action:
conjugating $e\mapsto e+\tau p$ by a map with common linear part $a$ in the two
coordinates gives $\tau\mapsto a\tau a^{-1}$. Its affine correction cancels.

\begin{lemma}\label{lem:frameperfect}
The group $\overline H$ is perfect.
\end{lemma}
\begin{proof}
First,
\[
 \cH_7=\langle\mathbf1\rangle\oplus
 \{(\ell(p))_{p\ne0}:\ell\in E^*\}.
\]
The linear-function code has dimension three and even weights; it is
contained in $\cH_7$ because the sum of any polynomial of degree at most two
on $E$ is zero in $\F_2$. The all-one word has odd weight. Thus $\overline E$ is the
dual natural module. Differences $a\ell-\ell$ span that module:
elementary transvections give each coordinate functional as such a
difference. Consequently $[\overline H,\,\overline E]=\overline E$.

By Lemma~\ref{lem:translationspan}, the conjugation differences
$a\tau a^{-1}-\tau$ span the trace-zero matrices, which map
isomorphically onto $M_3(\F_2)/\langle\id\rangle$.
It follows that the image of $[\overline H,\,\overline{P}_0]$ on that
quotient is the whole quotient. Together with $\overline E\leq\overline H'$, this
gives $\overline{P}_0\leq\overline H'$.

Finally $\operatorname{GL}_3(2)$ is perfect. Its elementary transvections
generate it, and for pairwise distinct $i$, $j$, $k$,
\[
 [\id+E_{ik},\,\id+E_{kj}]=\id+E_{ij}.
\]
Hence the image of $\overline H'$ on the last quotient of
\eqref{eq:projectivefiltration} is also the whole quotient, proving
$\overline H'=\overline H$.
\end{proof}
We also need the simplicity of the row group.
\begin{lemma}\label{lem:GL32simple}
The group $Q=\operatorname{GL}_3(2)$ is simple.
\end{lemma}
\begin{proof}
The action on the seven nonzero vectors is faithful and primitive: any two
distinct vectors are independent, so the action is doubly transitive.
For $v\ne0$ let
\[
 A_v=\{\id+v\ell:\ell\in E^*,\ \ell(v)=0\}.
\]
This is an abelian subgroup of order four, normalized by $Q_v$.
Its conjugates generate $Q$, since they include the elementary transvections.
If $N\lhd Q$ is nontrivial, faithfulness and primitivity make $N$ transitive,
so $Q=NQ_v$. In $Q/N$ the images of all conjugates of $A_v$ are consequently
the same abelian subgroup. They generate $Q/N$, making it abelian.
The preceding transvection commutator identity proves that $Q$ is perfect;
therefore $Q/N=1$.
\end{proof}

\subsection{The normal-subgroup argument}\label{sec:simpcriterion}
\begin{theorem}[Simplicity of the scalar quotient]\label{thm:simplicity}
The projective group $R=\langle H,\,r_*\rangle/\mu_4$ is nonabelian simple, of order
\[
 |R|=145\,926\,144\,000
     =2^{14}3^3 5^3\cdot7\cdot13\cdot29.
\]
\end{theorem}
\begin{proof}
Let $1\ne N\lhd R$. By Corollary~\ref{cor:crosskernel} and
Proposition~\ref{prop:primitive}, $N$ is transitive on crosses, so $R=NR_x$ and
\[
 |R/N|\mid |R_x|=35\,942\,400=2^{12}3^3 5^2\cdot13.
\]
Write tildes for images in $R/N$. The quotient
$\widetilde H/\widetilde{P_0}$ is an image of the simple group
$\overline H/\overline{P}_0\cong\operatorname{GL}_3(2)$.
Its order is prime to $7$, so it is trivial. Hence $\widetilde H$ is a
$2$-group. It is also perfect by Lemma~\ref{lem:frameperfect}; a finite
perfect $2$-group is trivial, since finite $2$-groups are solvable by
induction on their order using their nontrivial centres. Thus
$\overline H\le N$.

Now $R/N$ is generated by the image of the involution $r_*$. Its order also divides
\[
 [R:\overline H]=424125=3^2 5^3\cdot13\cdot29,
\]
which is odd. Thus $N=R$. The order theorem shows that $R$ has composite
order, so it is nonabelian.
\end{proof}
\begin{corollary}
The centre of $G$ is exactly $\mu_4$.
\end{corollary}
\begin{proof}
A central element of $G$ has central image in the nonabelian simple quotient
$R$, so it lies in $\mu_4$. The reverse inclusion holds by construction.
\end{proof}

\section{The Tits group}\label{sec:titsmain}
Put
\[
 R_x=G_x/Z\le R,\qquad \overline H_x=H_x/Z,\qquad \overline H_{x,\,y}=H_{x,\,y}/Z,
\]
where $Z=\mu_4$, $x=X_1$, $y=X_2$. The preceding order calculation gives
\begin{equation}\label{eq:localorders}
 |R_x|=35\,942\,400=2^{12}3^3 5^2\cdot13,\quad
 |\overline H_x|=12\,288,\quad |\overline H_{x,\,y}|=4096,\quad |Y|=1755.
\end{equation}
We shall prove that this cross stabilizer has a simple derived subgroup
of index two.

\subsection{Generators and a small normal quotient of the monomial stabilizer}
Define
\begin{equation}\label{eq:localgens}
 q=P^{-1}VJ_2 P,\qquad n=\rho^4.
\end{equation}
Each fixes $x$. The row permutations of $R_0$, $\rho$, $\beta$ are
\[
 (1\ 2)(3\ 6),\qquad (1\ 6\ 5)(2\ 4\ 3),\qquad
 (1\ 6\ 3\ 2)(4\ 5),
\]
respectively; row $0$ is fixed. On the three lines through $p_0$, the first two induce a transposition
and a three-cycle. The square of their product and its two conjugates by
$\rho$ induce, on the rows,
\[
 (2\ 6)(4\ 5),\qquad(1\ 3)(4\ 5),\qquad(1\ 3)(2\ 6).
\]
These form the nonidentity elements of the four-group kernel. Thus the
row image is all of $\operatorname{GL}_3(2)_{p_0}\cong S_4$, of order
$4\cdot6=24$. Direct multiplication gives
\begin{equation}\label{eq:smallorders}
 R_0^2=U^2=V^2=W^2=\id,\qquad \rho^{12}=q^4=n^3=\id.
\end{equation}

\begin{lemma}\label{lem:localsmallquot}
Let $N_0$ be the normal closure of $\rho^4Z$ in $\overline H_x=H_x/Z$. Then
$\overline H_x/N_0$ is an abelian group of exponent at most two, generated by the
images of $R_0$ and $q$. The images of $U$, $V$, $W$, $\rho$ and $E_C$ are trivial.
\end{lemma}
\begin{proof}
Write $p=p_0$ and $E_p=E/\langle p\rangle$. The translation factor of
the projective row kernel is
\[
 \{\tau\in\operatorname{End}(E):\tau p\in\langle p\rangle\}
                  /\langle \id\rangle.
\]
It has dimension six. Choose a splitting $E=\langle p\rangle\oplus E_p$ and
subtract the scalar part of $\tau$ on $\langle p\rangle$. Representatives
then have the form
\begin{equation}\label{eq:pointmoduleblocks}
 \tau=\begin{pmatrix}0&\ell\\0&B\end{pmatrix},
 \qquad \ell\in E_p^*,\ B\in\operatorname{End}(E_p).
\end{equation}
The functional
\begin{equation}\label{eq:translationhyperplane}
 \tau\longmapsto\operatorname{tr}(\overline \tau),\qquad
 \overline \tau\in\operatorname{End}(E/\langle p\rangle),
\end{equation}
is well-defined modulo $\id$, since $\dim E_p=2$.

The image of $n=\rho^4$ in the row group $S_4$ is a three-cycle, whose
normal closure is $A_4$. Let $\overline n$ denote its irreducible order-three
action on $E_p$. The $A_4$-commutators in \eqref{eq:pointmoduleblocks}
contain all of $E_p^*$ and the two-dimensional space
$[\overline n,\,\operatorname{End}(E_p)]$. Indeed $\overline n$ has no fixed vector on
$E_p^*$; conjugation on the endomorphisms has centralizer $\F_4$,
and averaging over $\langle\overline n\rangle$ gives its complementary
commutator space. The remaining trace-zero direction is the scalar
$\id_{E_p}$ in this centralizer.

Two further identities supply that direction and the sign kernel. Put $u=(nR_0)^4$. In the quotient basis represented
by $1$, $2$ modulo $p_0=6$, its translation induces
\begin{equation}\label{eq:localnormalcolumns}
 \overline \tau_u=\begin{pmatrix}1&1\\1&1\end{pmatrix},
 \quad\overline n=\begin{pmatrix}1&1\\1&0\end{pmatrix},
 \quad
 \overline \tau_u+\overline n\overline \tau_u\overline n^{-1}
             +\overline n^2\overline \tau_u\overline n^{-2}=\id_{E_p}.
\end{equation}
It belongs to the normal closure since setting $n=1$ gives $R_0^4=1$.
Its translation on $E$ has columns $(5,\,3,\,3)$; thus the first equality
is obtained simply by reducing these columns modulo $p_0$.
The second identity, a multiplication of monomial involutions, is
\begin{equation}\label{eq:signnormalmasks}
 [[n,\,V],\,V]Z=E_4Z.
\end{equation}
The commutators $[n,\,\overline E]$ give the two-dimensional
annihilator of $p_0$ in $E^*$. Since $4\cdot p_0=1$, the last identity
gives its remaining direction. Hence $\overline E\subseteq N_0$.
It follows from the preceding module calculation that $N_0$ contains
the full preimage of the trace-zero subspace in the projective row kernel.

The quotient row group is at most $S_4/A_4\cong\Z/2\Z$, represented by
$R_0$. Its remaining row kernel has order at most two, measured by
$\operatorname{tr}(\overline \tau)$ and represented by $q$: its columns $(2,\,4,\,2)$ give trace one.
This kernel is central, and $R_0^2=1$, so the quotient is elementary
abelian of order at most four. The maps $U$, $V$, $W$ induce respectively
zero and the two off-diagonal matrix units on $E_p$, hence have trace
zero and vanish in the quotient. Finally $\rho^3$ has translation columns
$(1,\,2,\,2)$, inducing $\id_{E_p}$ and again trace zero. Since $\rho^4=n$ also
vanishes, so does $\rho$. This proves all the assertions.
\end{proof}

Corollary~\ref{cor:localgenerators} gives, on passing to the projective groups,
\begin{equation}\label{eq:Jgenerated}
 R_x=\langle \overline H_x,\,tZ\rangle=\langle \overline H_x,\,gZ\rangle,\qquad
 R_{x,\,y}=\langle \overline H_{x,\,y},\,tZ\rangle,\qquad [R_{x,\,y}:\overline H_{x,\,y}]=5.
\end{equation}

\subsection{The local action on 1755 crosses}
\begin{proposition}\label{prop:localrankfive}
The action of $R_x$ on $Y$ is faithful and primitive. Its subdegrees are
\[
 1,\,\quad10,\,\quad80,\,\quad640,\,\quad1024.
\]
\end{proposition}
\begin{proof}
Equation~\eqref{eq:Jgenerated} identifies $R_{x,\,y}$ with the image of
$G_0=\langle H_{x,\,y},\,t\rangle$. Its suborbits were proved in
Proposition~\ref{prop:couplelocalorbits}.

Faithfulness follows from Lemma~\ref{lem:shellkernel}.

For primitivity, a block containing $y$ is a union of its stabilizer orbits
and its size divides $1755$. The largest proper divisor of $1755$ is $585$.
Thus a proper block can contain neither the $640$- nor the $1024$-orbit.
Its possible sizes are $1$, $11$, $81$, $91$, of which only $1$ divides $1755$.
\end{proof}

\subsection{The normal-subgroup bound}
\begin{lemma}\label{lem:localnormalindex}
Every nontrivial normal subgroup $N$ of $R_x$ has index at most two.
\end{lemma}
\begin{proof}
Primitivity makes $N$ transitive on $Y$, so
\[
 R_x=NR_{x,\,y},\qquad |R_x/N|\mid |R_{x,\,y}|=20480=2^{12}\cdot5.
\]
Thus $n=\rho^4$ and its normal closure $N_0$ in $\overline H_x$ have
trivial image. Lemma~\ref{lem:localsmallquot} kills $U$, $V$, $W$, $\rho$, $E_C$
and makes the image of $\overline H_x$ elementary abelian, generated
by $R_0$, $q$. Since $g^2=1$, the identity $t=gVgW$ also has trivial
image. By \eqref{eq:Jgenerated}, $R_x/N$ is therefore the image of
$\overline H_{x,\,y}$ and has order at most four.

To remove the possible second factor of order two, use $\rho^3$, whose
image is trivial by Lemma~\ref{lem:localsmallquot}. Its translation
matrix $\tau$ has columns $(1,\,2,\,2)$, as recorded in that lemma's proof. Thus
$\tau p_0=0$, so it fixes all four crosses of row $0$.
The values $\tau(2)=\tau(4)=2$ lie in $E_{x,\,y}=\{0,\,2,\,4,\,6\}$,
so $\rho^3$ fixes every quartet of $\Omega$. In row $4$ of binary label
$4$, translation by $2$ swaps the two crosses of each complementary
couple; in particular it swaps $X_{18}$, $X_{20}$ in $Q_3$.

Equation~\eqref{eq:onebridge} and Lemma~\ref{lem:quartetexhaustion} give
$tQ_1=Q_3$: the image contains $x$, $y$, $X_{20}$, and only $Q_3$ does so.
Consequently $t^{-1}\rho^3t$ fixes $x$, $y$ and swaps $X_3$, $X_4$. It stabilizes
$Q_1$ and therefore its canonical frame $F$, so belongs to $H_{x,\,y}$.
Since it fixes $x$, its affine offset in row $0$ is in $\langle p_0\rangle$.
Its row action therefore induces a transposition on the three nonzero
cosets of $\langle p_0\rangle$, or on the three lines through $p_0$.
It is thus odd in the row group $S_4$ and has image $R_0q^\epsilon$
for some $\epsilon\in\{0,\,1\}$ in the quotient of
Lemma~\ref{lem:localsmallquot}. But $(t^{-1}\rho^3t)Z\in N$ by normality.
This relation leaves at most one generator of order two, proving
$[R_x:N]\le2$.
\end{proof}

We next construct a nontrivial character, and hence a subgroup of
index two, by the local action on the five quartets.

\subsection{The local affine action and transfer}
The action of $R_{x,\,y}$ on $\Omega$ is transitive with point stabilizer
$\overline H_{x,\,y}$, a $2$-group. Its faithful image has order $5\cdot2^a$ with
$a\le3$, so the number of its Sylow $5$-subgroups divides eight and is
$1\pmod5$, hence is one. It therefore lies in the normalizer of a
$5$-cycle, namely $\operatorname{AGL}_1(5)$. The row map $\beta$ acts
as a four-cycle on $\Omega\setminus\{Q_1\}$, by
\eqref{eq:couplemonomialgens} and Table~\ref{tab:hamminglabels}. Thus
\begin{equation}\label{eq:affinefive}
 R_{x,\,y}^\Omega=\operatorname{AGL}_1(5)=5:4.
\end{equation}
Conjugation on its normal cyclic subgroup of order five defines
$m:R_{x,\,y}\to\F_5^\times$. Every nontrivial involution in the five-point
image has multiplier $-1$.

Transport gives the same multiplier homomorphism on each point stabilizer
$R_{x,\,v}$, $v\in Y$. It is independent of the transporter: the target is abelian,
so conjugating within $R_{x,\,y}$ does not change $m$. Define
\begin{equation}\label{eq:transferelementary}
 \nu(a)=\prod_{v\in Y}m(r_{av}^{-1}a r_v),
 \qquad r_vy=v.
\end{equation}
This is a homomorphism $R_x\to\F_5^\times$, because the factors for $ab$
are the products of the factors for $a$ at $bv$ and for $b$ at $v$.
Changing the representatives cancels factors. Equivalently, a cycle of $a$
on $Y$ of length $s$, represented by $v$, contributes the local multiplier
$m_v(a^s)$. This is the elementary transfer formula.

\begin{lemma}\label{lem:nontrivialtransfer}
Put $\alpha=qW\in P_{0,\,x}$.  The homomorphism in
\eqref{eq:transferelementary} satisfies $\nu(\alpha Z)=-1$.
\end{lemma}
\begin{proof}
The translation columns of $q$ and $W$ are $(2,\,4,\,2)$ and $(0,\,1,\,1)$;
those of $\alpha$ are therefore $(2,\,5,\,3)$.  Direct multiplication gives
\[
 \alpha^4=\id,\qquad
 \alpha^2=E_3,
\]
where $E_3\in E_C$ acts as $(-1)^{3\cdot p}$ on row $p$.
Thus $\alpha^2$ negates coordinate rows $0$, $1$, $2$, $5$, the affine plane
$3\cdot p=1$, and fixes rows $3$, $4$, $6$.

A cross fixed by $\alpha^2$ has its support in one of these two eigenspaces.
No noncoordinate minimum is supported on three rows.  The four-block
crosses on the four negative rows have label $3$, with $3\cdot p_0=1$,
and so are nonorthogonal to $x$.  A heavy cross uses all seven rows.
Consequently the fixed points of $\alpha^2$ on $Y$ are exactly the $27$
coordinate crosses other than $x$.

Let $\tau$ have columns $(2,\,5,\,3)$.  Then $\tau p_0\in\langle p_0\rangle$, whereas
$\tau p\notin\langle p\rangle$ for every nonzero $p\ne p_0$.
Hence $\alpha$ fixes $X_2$, $X_3$, $X_4$, has two two-cycles in each of the
other six coordinate rows, and has only four-cycles on the remaining points
of $Y$.  Every four-cycle contributes $1$ to the transfer.
The sign change $\alpha^2\in E_C$ is central in $P_0$.
Point-fixing translations can take any prescribed value in another row,
so $P_{0,\,x}$ is transitive on its four coordinate crosses.
Equivariance of the transported homomorphisms therefore makes
$m_v(\alpha^2Z)$ constant on each row other than row $0$.  This value is $\pm1$, so the
two two-cycles in a row contribute its square, namely $1$.

It remains to evaluate the three fixed points.  Transport the five four-sets
$\Omega$ by $1$, $\GT$, $\GT^2$ to the couples $\{x,\,X_2\}$, $\{x,\,X_3\}$,
$\{x,\,X_4\}$, using $\GT\in H_x$, which sends $X_2\mapsto X_3\mapsto X_4$.
Their direction planes and the translation values on the
other two rows of the corresponding dozen are
\[
\begin{array}{c|c|c|c}
\text{fixed point}&\text{direction plane }E_{x,\,X_j}&\text{other row labels }p&\tau p\\\hline
 X_2&\{0,\,2,\,4,\,6\}&2,\,4&5,\,3\\
 X_3&\{0,\,1,\,6,\,7\}&1,\,7&2,\,4\\
 X_4&\{0,\,3,\,5,\,6\}&3,\,5&7,\,1.
\end{array}
\]
These are obtained from Table~\ref{tab:hamminglabels}: the plane $E_{x,\,y}$ is
mapped by the linear action of $\GT$ on $E$, with column triple $(5,\,7,\,1)$,
to $\{0,\,1,\,6,\,7\}$ and then to $\{0,\,3,\,5,\,6\}$. In each of the two
other rows, the two coordinate couples
correspond to the two cosets of $E_{x,\,X_j}$ in $E$.  A translation swaps them
exactly when its value lies outside $E_{x,\,X_j}$; adjoining the initial couple
then swaps the corresponding members of $\Omega$.
All six displayed translation values lie outside their direction planes.
Thus $\alpha$ induces a nontrivial double transposition on each of the
three transported five-sets, with multiplier $-1$.  The transfer formula gives
\[
 \nu(\alpha Z)=(-1)^3=-1.
\]
\end{proof}

\subsection{The simple derived subgroup}
\begin{theorem}\label{thm:localTits}
The group $R_x'=(G_x/Z)'$ is nonabelian simple, with
\[
 [R_x:R_x']=2,\qquad |R_x'|=17\,971\,200=2^{11}3^3 5^2\cdot13.
\]
It is the unique proper nontrivial normal subgroup of $R_x$, and equals
$\ker\nu$ for the explicit character \eqref{eq:transferelementary}.
\end{theorem}
\begin{proof}
Put $M=\ker\nu$. It is proper by Lemma~\ref{lem:nontrivialtransfer}, and
nontrivial because $|R_x|>4=|\F_5^\times|$. Lemma~\ref{lem:localnormalindex}
therefore gives $[R_x:M]=2$ and the displayed order for $M$.
The element $\alpha Z$ lies outside $M$, so $R_x=\langle M,\,\alpha Z\rangle$.

Suppose $1\ne A\lhd M$ is proper, and put
$A_1=(\alpha Z)A(\alpha Z)^{-1}$. Both $A$ and $A_1$ are normal in $M$.
Since $(\alpha Z)^2\in M$, conjugation by $\alpha Z$ interchanges them. Consequently
$A\cap A_1$ and $AA_1$ are normal in $R_x$. A nontrivial intersection
would have order at least $|R_x|/2=|M|$ by
Lemma~\ref{lem:localnormalindex}, contrary to $A<M$; hence
$A\cap A_1=1$. The product is nontrivial and contained in $M$, so the
same lemma gives $AA_1=M$. Their commutator lies in their intersection,
and therefore
\[
 M=A\times A_1,\qquad |M|=|A|^2.
\]
This is impossible because $13$ occurs to exponent one in $|M|$.
Thus $M$ is simple, and its composite order makes it nonabelian.
It is consequently perfect. Since $R_x/M$ is abelian,
$M=M'\le R_x'\le M$, proving $R_x'=M$.

Finally, a nontrivial normal subgroup $N$ of $R_x$ has order at least
$|M|>2$. Its intersection with $M$ cannot be trivial, since then
$N$ would embed in $R_x/M$. Simplicity of $M$ gives $M\le N$, so
$N=M$ or $R_x$. This proves uniqueness.
\end{proof}

\begin{proof}[Completion of the proof of Theorem~\ref{thm:groupintro}]
It remains to compare the vector and cross stabilizers. For $a\in G_x$,
write $af=\chi(a)f$. The homomorphism $\chi:G_x\to\mu_4$ is the
identity on the central subgroup $\mu_4$ and has kernel $G_f$. Hence
$G_x=\mu_4\times G_f$ and $G_f\cong G_x/\mu_4=R_x$; taking derived
subgroups gives $G_x'=G_f'\cong R_x'$. The remaining assertions are
Corollary~\ref{cor:canonicalcount}, Theorem~\ref{thm:couplereflections},
Corollary~\ref{cor:intrinsicgenerator}, Theorems~\ref{thm:order}
and~\ref{thm:simplicity}, the centre corollary, and
Theorem~\ref{thm:localTits}.
\end{proof}
Under the classical-lattice identification of
Remark~\ref{rem:identify}, these groups are
$R\cong\Ru$, $R_x\cong{}^2F_4(2)$, and
$R_x'\cong{}^2F_4(2)'$ \cite[\S\S1.2,1.8]{Wilson}.
Wilson's absolute vector stabilizer is taken in the double cover of $\Ru$;
our $G_f$ is taken in the full Gaussian isometry group and has twice the
order of $R_x'$.

\section*{Related constructions and historical notes}
Ree constructed the groups ${}^2F_4(q)$ for $q=2^{2m+1}$, simple for $m\ge1$
\cite{Ree}. Tits proved that the exceptional smallest group ${}^2F_4(2)$
has a simple derived subgroup of index two \cite{Tits}. In the classification of finite simple groups (CFSG), the
Tits group is conventionally included with the groups of Lie type, rather
than among the twenty-six sporadic groups \cite{WilsonFSG}. Parrott and
Stroth gave characterizations \cite{ParrottTits,StrothTits}; Wilson determined
its maximal subgroups \cite{Wilson}. Wilson
also gave an elementary construction and order calculation for the Ree
family, with a simplicity proof for $q>2$ \cite{WilsonRee}.

Local characterizations of $\Ru$ were obtained by Dempwolff, Parrott, Assa,
and O'Nan \cite{Dempwolff,ParrottRu,Assa,ONan}. Rudvalis's detailed rank-three
and character-theoretic treatment appeared in 1984 \cite{Rud84II,Rud84I},
as did Wilson's geometric and maximal-subgroup analysis \cite{Wilson}.
In the CFSG, $\Ru$ is one of the twenty-six sporadic groups
\cite{WilsonFSG}; it is also one of the six pariahs, not involved in the
Monster \cite{Griess}.

Other constructions include that of Stroth and Weiss \cite{SW} and
Coolsaet's construction from the automorphism group $U_3(5){:}2$
of the Hoffman--Singleton graph \cite{Coolsaet}.
Conway had already exhibited a $50$-vector Hoffman--Singleton
configuration inside the quaternionic model \cite[\S7, Table~5]{Conway}.
Coolsaet constructs a rank-three group on $4060$ vertices without computer
calculation, then identifies it as $\Ru$ using the classification of
rank-three groups \cite[\S7, Theorem~3]{Coolsaet}.
Chigira, reporting joint work with Kitazume, gives a lattice construction
based on the $28$-point action of $U_3(3)$ \cite{Chigira}, while Duncan realizes cyclic covers of
$\Ru$ as full automorphism groups of enhanced vertex operator superalgebras
\cite{DuncanI,DuncanII}.

The exceptional-group approach includes the work of Griess and Ryba on
embeddings in type $E_7$ \cite{GR} and the construction of
$\Ru<E_7(5)$ by Kleidman, Meierfrankenfeld, and Ryba \cite{KMR}.
The invariant alternating product on the $133$-dimensional module in
characteristic five, suggested in the latter work and verified to be a Lie
product by Craven \cite[\S5.2.2]{Craven}, gives another proof of the
embedding.
Bradley, Curtis, and Malik give symmetric presentations using $105$
involutions with control group $L_4(2)$ or seven with control group $L_3(2)$;
in the latter realization inside $\Ru$, four generators indexed by a Fano-plane
quadrangle generate the Tits simple group \cite[Abstract]{BCM}.

\end{document}